\documentclass[11pt]{article}
 \usepackage{epsfig}
 \usepackage{amssymb,amsmath,amsthm,amscd}
 \usepackage{latexsym}

\newtheorem{thm}{Theorem}[section]
\newtheorem{prop}[thm]{Proposition}

\newtheorem{lem}[thm]{Lemma}

\theoremstyle{definition}
\newtheorem{defn}[thm]{Definition}

\newcounter{labelflag} 

\newcommand{\Label}[1]{
                       \ifnum\thelabelflag=1
                          \ifmmode
                             \makebox[0in][l]{\qquad\fbox{\rm#1}}
                          \else
                             \marginpar{\vspace{0.7\baselineskip}
                                        \hspace{-1.1\textwidth}
                                        \fbox{\rm#1}}
                          \fi
                       \fi
                       \label{#1}
                      }

\newcommand{\be}{\begin{equation}}
\newcommand{\ee}{\end{equation}}

\newcommand{\vt}{{\tilde{v}}}
\newcommand{\ut}{{\tilde{u}}}

 \newcommand{\R}{\mathbb{R}}

\newcommand{\hone}{ H^1(\mathbb{R}^3) }
\newcommand{\ltwo}{ L^2(\mathbb{R}^3) }

\begin{document}

\begin{titlepage}
\title{\Large\bf   Asymptotic Behavior of Stochastic Wave
Equations with Critical Exponents on $\R^3$}
\vspace{7mm}

\author{
Bixiang Wang  \thanks {Supported in part by NSF  grant DMS-0703521}
\vspace{1mm}\\
Department of Mathematics, New Mexico Institute of Mining and
Technology \vspace{1mm}\\ Socorro,  NM~87801, USA \vspace{1mm}\\
Email: bwang@nmt.edu}
\date{}
\end{titlepage}

\maketitle

\begin{abstract}
The
existence of a random attractor
in $H^1(\mathbb{R}^3) \times L^2(\mathbb{R}^3)$ is proved
for the  damped semilinear stochastic wave equation
defined on the entire space $\mathbb{R}^3$. The nonlinearity is
allowed to have a cubic growth rate which is referred to as the
critical exponent. The uniform pullback estimates
on the tails of solutions  for large space variables are  established.
The pullback asymptotic compactness of the random dynamical system
is proved  by using these  tail estimates and the energy equation method.
\end{abstract}

{\bf Key words.} Random attractor, asymptotic compactness,
 wave equation.

{\bf MSC 2000.}  37L55; Secondary: 60H15, 35B40.

\baselineskip=1.5\baselineskip

 \section{Introduction}
\setcounter{equation}{0}

This paper deals with the existence of a random attractor
 for the  stochastic wave  equation defined  on $\mathbb{R}^3$:
\begin{equation}
\label{intro1}
u_{tt} + \alpha u_t -\Delta u + \lambda u
+ f(x, u) = g(x) + h(x) {\frac {dw}{dt}},
\end{equation}
  with the initial
 conditions
 \begin{equation}
 \label{intro2}
 u(x, \tau) =u_0(x), \quad u_t(x, \tau) = u_1(x),
 \end{equation}
 where
 $x \in \mathbb{R}^3$, $   t> \tau
 $ with $   \tau \in \mathbb{R}$,
 $  \alpha$ and $\lambda$
  are     positive numbers,
 $g  $ and $h$  are given in $L^2(\mathbb{R}^3) $ and
 $ H^1(\mathbb{R}^3)$ respectively,
$f$ is a nonlinear  function with cubic growth
rate  (called  the critical exponent), and
 $w $   is  an independent   two-sided real-valued
 Wiener process  on a probability
 space.

The global attractors of deterministic wave equations
have been studied extensively in the literature, see,
e.g., \cite{bab1, bal1, hal1, sel1, tem1} and the references therein.
Particularly,  the existence of  attractors
was   proved
in \cite{arr1, bab1,  bal1, chu1,  fei4, kha1,   sun1, sun2}
for the deterministic wave  equations defined on
{\it bounded domains}
with    critical
  exponents,
and in \cite{fei1, fei2, fei3, pri1}
for the   equations
defined  on {\it unbounded domains}
with critical or supercritical exponents.
In this paper, we will prove the existence of a {\it random attractor}
for the stochastic wave equation \eqref{intro1} with critical exponents
defined on the entire space $\R^3$.

The interesting features of problem \eqref{intro1}-\eqref{intro2}
lie in:  (i) The equation is stochastic.  In this case, problem
\eqref{intro1}-\eqref{intro2} determines a random dynamical system
instead of a deterministic  semigroup; (ii) The nonlinearity $f$ is
   critical.  The difficulty caused
 by the non-compactness of  embedding $H^1  \hookrightarrow L^6  $
must be overcome
in order to deal with the asymptotic compactness of solutions
with such a critical nonlinearity;
(iii) The domain  $\R^3$ of problem \eqref{intro1}-\eqref{intro2}
is unbounded. In this case,  the  embeddings $H^1(\R^3) \hookrightarrow
L^p(\R^3)$ are not compact  even for $p<6$.
This is essentially different from the case of bounded domains.

To study the long term behavior of solutions of stochastic
differential equations, the concept of
  random attractor should be used instead of global
attractor, which was introduced   in \cite{cra2, fla1} for random dynamical systems.
Since the nonlinearity  $f$ of equation \eqref{intro1}
has  a  critical growth rate,  the mapping $f$
from  $H^1(Q)$  to  $L^2(Q)$  is continuous, but not compact, even for
a bounded domain $Q$ in $\R^3$.
To circumvent the difficulty and prove the asymptotic compactness
of the deterministic wave equation on a bounded domain $Q$,
an energy equation approach was developed by Ball in \cite{bal1}.
This method is quite effective for a variety of applications,
see, e.g., \cite{bal2, ju1, moi1,moi2,   wanx}.
Notice that the compactness of embeddings
$H^1(Q) \hookrightarrow
L^p(Q)$ with $p<6$
 was crucial and  frequently used in \cite{bal1}
when $Q$ is {\it bounded}.
In our case, the domain $\R^3$ is unbounded, and hence
the  embeddings
  $H^1(\R^3) \hookrightarrow
L^p(\R^3)$ are not compact for any $p$.
This means  that   Ball's  method \cite{bal1}
alone is not sufficient for proving
the asymptotic compactness of the equation on $\R^3$.
We must overcome the
difficulty caused by the
  non-compactness of embeddings
$H^1(\R^3) \hookrightarrow
L^p(\R^3)$ for $p<6$.
In this paper, we will solve the problem
 by using
  the method of tail estimates
  developed in \cite{wan1}
for deterministic parabolic  equations.
In other words, we will first show that the solutions
of problem \eqref{intro1}-\eqref{intro2} uniformly
approach zero,  in a sense,  as $x$ and $t$ go to infinity, and then
apply these tail estimates and  the energy equation method \cite{bal1}
 to prove the asymptotic
compactness of the stochastic wave equations on $\R^3$.

The random attractors of stochastic equations have been
investigated by several authors
  in \cite{arn1, car1, car2, cra1, cra2, fla1}
and the references therein.
In these papers, the domains of PDEs  were supposed to be
{\it  bounded}.
In the case of {\it unbounded} domains, the existence of random attractors
has been established  recently for   parabolic and wave equations
in
   \cite{bat2}
 and \cite{wan2}, respectively.
Notice that the method of \cite{wan2} only works for
the wave equation with
subcritical nonlinearity, and is not valid for the
critical case. It is the intension of this paper
to prove the existence of  a random attractor
for the stochastic wave equation with critical nonlinearity on $\R^3$.

 This paper is organized  as follows. In the next section, we
 recall the
 random attractors  theory  for random dynamical systems. In Section 3,
 we define  a continuous  random dynamical system for problem \eqref{intro1}-\eqref{intro2}.
 The uniform estimates of solutions  are contained
 in Section 4, which include
   uniform estimates on the tails of solutions.
In Section 5, we  prove  the pullback  asymptotic compactness
and  the existence  of random attractors for the stochastic
wave equation on $\R^3$.

In the sequel,  we adopt  the following notations.  We denote by
$\| \cdot \|$ and $(\cdot, \cdot)$ the norm and the inner product
of
  $L^2(\mathbb{R}^3)$, respectively.    The
norm of  a given  Banach space $X$  is written as    $\|\cdot\|_{X}$.
We also use $\| \cdot\|_{p}$    to denote   the norm  of
$L^{p}(\mathbb{R}^3)$.  The letters $c$ and $c_i$ ($i=1, 2, \ldots$)
are  generic positive constants  which may change their  values from line to
line or even in the same line.

\medskip

\section{Preliminaries}
\setcounter{equation}{0}

In this section, we recall some basic concepts
related to random attractors for stochastic dynamical
systems. The reader is referred to \cite{arn1, bat1, cra1, fla1} for more details.

Let  $(X, \| \cdot \|_X)$ be a   separable
Hilbert space with Borel $\sigma$-algebra $\mathcal{B}(X)$,
 and
$(\Omega, \mathcal{F}, P)$  be  a probability space.

\begin{defn}
$(\Omega, \mathcal{F}, P,  (\theta_t)_{t\in \R})$
is called a  metric   dynamical  system
if $\theta: \R \times \ \Omega \to \Omega$ is
$(\mathcal{B}(\R) \times \mathcal{F}, \mathcal{F})$-measurable,
$\theta_0$ is the identity on $\Omega$,
$\theta_{s+t} = \theta_t \circ  \theta_s$ for all
$s, t \in \R$ and $\theta_t P = P$ for all $t \in \R$.
 \end{defn}

\begin{defn}
\label{RDS}
A continuous random dynamical system (RDS)
on $ X  $ over  a metric  dynamical system
$(\Omega, \mathcal{F}, P,  (\theta_t)_{t\in \R})$
is  a mapping
  $$
\Phi: \R^+ \times \Omega \times X \to X \quad (t, \omega, x)
\mapsto \Phi(t, \omega, x),
$$
which is $(\mathcal{B}(\R^+) \times \mathcal{F} \times \mathcal{B}(X), \mathcal{B}(X))$-measurable and
satisfies, for $P$-a.e.  $\omega \in \Omega$,

(i) \  $\Phi(0, \omega, \cdot) $ is the identity on $X$;

(ii) \  $\Phi(t+s, \omega, \cdot) = \Phi(t, \theta_s \omega,
\cdot) \circ \Phi(s, \omega, \cdot)$ for all $t, s \in \R^+$;

(iii) \  $\Phi(t, \omega, \cdot): X \to  X$ is continuous for all
$t \in  \R^+$.
\end{defn}

Hereafter, we always assume that $\Phi$  is a continuous RDS on $X$
over $(\Omega, \mathcal{F}, P,  (\theta_t)_{t\in \R})$.

\begin{defn}
 A random  bounded set $\{B(\omega)\}_{\omega \in \Omega}$
 of  $  X$  is called  tempered
 with respect to $(\theta_t)_{t\in \R}$ if for $P$-a.e. $\omega \in \Omega$,
 $$ \lim_{t \to \infty} e^{- \beta t} d(B(\theta_{-t} \omega)) =0
 \quad \mbox{for all} \  \beta>0,
 $$
 where $d(B) =\sup_{x \in B} \| x \|_{X}$.
\end{defn}

\begin{defn}
 A random   function  $r(\omega) $
   is called  tempered
 with respect to $(\theta_t)_{t\in \R}$ if for $P$-a.e. $\omega \in \Omega$,
 $$ \lim_{t \to \infty} e^{- \beta t} r (\theta_{-t} \omega ) =0
 \quad \mbox{for all} \  \beta>0 .
 $$
\end{defn}

\begin{defn}
Let $\mathcal{D}$ be a collection of  random  subsets of $X$.
Then  $\mathcal{D}$ is called inclusion-closed if
   $D=\{D(\omega)\}_{\omega \in \Omega} \in {\mathcal{D}}$
and  $\tilde{D}=\{\tilde{D}(\omega) \subseteq X:  \omega \in \Omega\} $
with
  $\tilde{D}(\omega) \subseteq D(\omega)$ for all $\omega \in \Omega$ imply
  that  $\tilde{D} \in {\mathcal{D}}$.
  \end{defn}

\begin{defn}
Let $\mathcal{D}$ be a collection of random subsets of $X$ and
$\{K(\omega)\}_{\omega \in \Omega} \in \mathcal{D}$. Then
$\{K(\omega)\}_{\omega \in \Omega} $ is called an absorbing set of
$\Phi$ in $\mathcal{D}$ if for every $B \in \mathcal{D}$ and
$P$-a.e. $\omega \in \Omega$, there exists $t_B(\omega)>0$ such
that
$$
\Phi(t, \theta_{-t} \omega, B(\theta_{-t} \omega)) \subseteq
K(\omega) \quad \mbox{for all} \ t \ge t_B(\omega).
$$
\end{defn}

\begin{defn}
Let $\mathcal{D}$ be a collection of random subsets of $X$. Then
$\Phi$ is said to be  $\mathcal{D}$-pullback asymptotically
compact in $X$ if  for $P$-a.e. $\omega \in \Omega$,
$\{\Phi(t_n, \theta_{-t_n} \omega,
x_n)\}_{n=1}^\infty$ has a convergent  subsequence  in $X$
whenever
  $t_n \to \infty$, and $ x_n\in   B(\theta_{-t_n}\omega)$   with
$\{B(\omega)\}_{\omega \in \Omega} \in \mathcal{D}$.
\end{defn}

\begin{defn}
Let $\mathcal{D}$ be a collection of random subsets of $X$
and $\{\mathcal{A}(\omega)\}_{\omega \in \Omega} \in  \mathcal{D}$.
Then   $\{\mathcal{A}(\omega)\}_{\omega \in \Omega} $
is called a   $\mathcal{D}$-random   attractor
(or $\mathcal{D}$-pullback attractor)  for
  $\Phi$
if the following  conditions are satisfied, for $P$-a.e. $\omega \in \Omega$,

(i) \  $\mathcal{A}(\omega)$ is compact,  and
$\omega \mapsto d(x, \mathcal{A}(\omega))$ is measurable for every
$x \in X$;

(ii) \ $\{\mathcal{A}(\omega)\}_{\omega \in \Omega}$ is invariant, that is,
$$ \Phi(t, \omega, \mathcal{A}(\omega)  )
= \mathcal{A}(\theta_t \omega), \ \  \forall \   t \ge 0;
$$

(iii) \ \ $\{\mathcal{A}(\omega)\}_{\omega \in \Omega}$
attracts  every  set  in $\mathcal{D}$,  that is, for every
 $B = \{B(\omega)\}_{\omega \in \Omega} \in \mathcal{D}$,
$$ \lim_{t \to  \infty} d (\Phi(t, \theta_{-t}\omega, B(\theta_{-t}\omega)), \mathcal{A}(\omega))=0,
$$
where $d$ is the Hausdorff semi-metric given by
$d(Y,Z) =
  \sup_{y \in Y }
\inf_{z\in  Z}  \| y-z\|_{X}
 $ for any $Y\subseteq X$ and $Z \subseteq X$.
\end{defn}

The following existence result on a    random attractor
for a  continuous  RDS
can be found in \cite{bat1,  fla1}.
\begin{prop}
\label{att} Let $\mathcal{D}$ be an inclusion-closed
 collection of random subsets of
$X$ and $\Phi$ a continuous RDS on $X$ over $(\Omega, \mathcal{F},
P,  (\theta_t)_{t\in \R})$. Suppose  that $\{K(\omega)\}_{\omega
\in K} $ is a closed  absorbing set of  $\Phi$  in $\mathcal{D}$
and $\Phi$ is $\mathcal{D}$-pullback asymptotically compact in
$X$. Then $\Phi$ has a unique $\mathcal{D}$-random attractor
$\{\mathcal{A}(\omega)\}_{\omega \in \Omega}$ which is given by
$$\mathcal{A}(\omega) =  \bigcap_{\tau \ge 0} \  \overline{ \bigcup_{t \ge \tau} \Phi(t, \theta_{-t} \omega, K(\theta_{-t} \omega)) }.
$$
\end{prop}

In this paper, we will
denote by $\mathcal{D}$
the collection of all tempered random sets
of $\hone \times \ltwo$, and     prove
problem \eqref{intro1}-\eqref{intro2}  has a $\mathcal{D}$-random attractor.

 \section{Random Dynamical Systems}
\setcounter{equation}{0}

In this section, we define a continuous
random dynamical system for
problem \eqref{intro1}-\eqref{intro2}.
Denote by
 $z= u_t + \delta u$ where $\delta$ is a small
 positive number to be determined later.
Substituting $u_t =z -\delta u$ into
\eqref{intro1} we find that
 \begin{equation}
 \label{spde1}
 {\frac {du}{dt}} + \delta u =z,
 \end{equation}
 \begin{equation}
 \label{spde2}
 {\frac {dz}{dt}} + (\alpha -\delta) z
 + (\lambda + \delta^2 -\alpha \delta) u
 -\Delta u + f(x, u) =g(x) + h(x) {\frac {dw}{dt}},
 \end{equation}
   with the initial
 conditions
 \begin{equation}
 \label{spde3}
 u(x, \tau) =u_0(x), \quad z(x, \tau) = z_0(x),
 \end{equation}
  where $z_0(x) = u_1(x)+ \delta u_0(x)$,
 $x \in \mathbb{R}^3$, $   t> \tau
 $ with $   \tau \in \mathbb{R}$,
 $  \alpha$ and $\lambda$
  are     positive numbers,
 $g \in L^2(\mathbb{R}^3) $ and $h \in H^1(\mathbb{R}^3)$  are  given, and
 $w $   is  an independent   two-sided real-valued
 Wiener process  on a complete probability
 space $(\Omega, \mathcal{F}, P)$
 with path $\omega(\cdot)$
 in $C(\mathbb{R}, \mathbb{R})$
 satisfying $\omega(0) =0$.
In addition, $(\Omega, \mathcal{F}, P, (\theta_t)_{t\in\mathbb{R}})$
 forms   a metric dynamical system, where
 $(\theta_t)_{t\in\mathbb{R}}$ is a
   family of measure preserving
 shift operators
   given  by
 $$
 \theta_t \omega (\cdot) = \omega (\cdot +t)
 - \omega (t), \quad  \forall \   \omega \in \Omega
\ \  \mbox{and} \  \  t\in \mathbb{R}.
 $$
 Let   $F(x,u) = \int_0^u f(x,s) ds$ for $x\in \mathbb{R}^3$ and $u\in \mathbb{R}$.
We assume the following conditions on the  the nonlinearity $f$,
  for every
 $x\in \mathbb{R}^3$ and $u\in \mathbb{R}$,
 \begin{equation}
 \label{f1}
 |f(x,u)|
 \le c_1 |u|^\gamma + \phi_1(x),\quad \phi_1 \in L^2(\mathbb{R}^3),
 \end{equation}
 \begin{equation}
 \label{f2}
 f(x,u) u - c_2 F(x, u) \ge \phi_2(x),
 \quad \phi_2 \in L^1(\mathbb{R}^3),
 \end{equation}
 \begin{equation}
 \label{F2}
 F(x,u) \ge c_3 |u|^{\gamma +1} -\phi_3,
 \quad \phi_3 \in L^1(\mathbb{R}^3),
 \end{equation}
 \begin{equation}
 \label{f3}
 |f_u (x,u)| \le c_4 |u|^{\gamma -1} + \phi_4,
 \quad \phi_4 \in H^1(\mathbb{R}^3),
 \end{equation}
 where  $1\le \gamma \le3$.
As a special case,   $\gamma =3$ is referred to as
the critical exponent.
 Notice that \eqref{f1}
 and \eqref{f2} imply
 \begin{equation}
 \label{F3}
 F(x, u) \le c (|u|^2 + |u|^{\gamma +1} + \phi_1^2 + \phi_2),
 \end{equation}
which is useful when deriving uniform estimates
of solutions.

To study the dynamical behavior of problem
 \eqref{spde1}-\eqref{spde3}, we need to convert
the stochastic system
 into a deterministic one with a random
 parameter.
To this end, we set $v(t, \tau, \omega) = z(t, \tau, \omega)
 - h \omega(t).$ Then it follows from
   \eqref{spde1}-\eqref{spde3} that
 \begin{equation}
 \label{pde1}
 {\frac {du}{dt}} + \delta u -v  = h\omega (t),
 \end{equation}
 \begin{equation}
 \label{pde2}
 {\frac {dv}{dt}} + (\alpha -\delta) v
 + (\lambda + \delta^2 -\alpha \delta) u
 -\Delta u + f(x, u) =g + (\delta -\alpha) h \omega (t),
 \end{equation}
   with the initial
 conditions
 \begin{equation}
 \label{pde3}
 u(x, \tau) =u_0(x), \quad v(x, \tau) = v_0(x),
 \end{equation}
 where $v_0(x) = z_0(x) -h\omega(\tau)$.

 By a standard method as in \cite{fei1}, it can be proved
 that  problem \eqref{pde1}-\eqref{pde3}
 with \eqref{f1}-\eqref{f3}
 is well-posed in $H^1(\mathbb{R}^3) \times L^2(\mathbb{R}^3)$,
 that is, for $P$-a.e. $\omega \in \Omega$, for every $\tau \in \mathbb{R}$ and
 $(u_0, v_0) \in H^1(\mathbb{R}^3) \times L^2(\mathbb{R}^3)$,
 problem \eqref{pde1}-\eqref{pde3}
 has a unique solution
 $(u(\cdot, \tau, \omega), v(\cdot, \tau, \omega))
 \in C([\tau, \infty), H^1(\mathbb{R}^3) \times L^2(\mathbb{R}^3))$
 with $(u(\tau, \tau, \omega), v(\tau, \tau, \omega))
 =(u_0, v_0)$. Further, the solution is continuous with respect
 to   $(u_0, v_0)$ in $H^1(\mathbb{R}^3) \times L^2(\mathbb{R}^3)$.
Sometimes, we also write the solution as
$(u(t, \tau, \omega, u_0), v(t, \tau, \omega, v_0))$ to indicate
 the dependence of $(u,v)$ on initial data $(u_0, v_0)$.
The following weak continuity of solutions on initial data is useful
when proving the asymptotic compactness of solutions
in the last section.

\begin{lem}
\label{weak_cont}
Assume that $g \in L^2(\mathbb{R}^3)$, $ h  \in  H^1(\mathbb{R}^3)$ and
 \eqref{f1}-\eqref{f3} hold.
Then the solution $(u,v)$ of problem
\eqref{pde1}-\eqref{pde3} is weakly continuous with respect to
initial data $(u_0, v_0)$ in $\hone \times \ltwo$. That is, for
  $P$-a.e. $\omega \in \Omega$, $\tau \in \R$ and $t \ge \tau$,
$(u(t, \tau, \omega, u_{0,n}), v(t, \tau, \omega, v_{0,n}))$
weakly converges to
$(u(t, \tau, \omega, u_{0 }), v(t, \tau, \omega, v_{0 }))$
in $\hone \times \ltwo$ provided $(u_{0,n}, v_{0,n})$
weakly converges to $(u_0, v_0)$ in
$\hone \times \ltwo$.
\end{lem}
\begin{proof}
 The proof is quite standard (see, e.g., \cite{bal1}) and hence
omitted here.
\end{proof}

We now define a random dynamical system for the
stochastic wave equation. Let $\Phi$ be a mapping,
 $\Phi$: $\mathbb{R}^+ \times \Omega \times
 H^1(\mathbb{R}^3) \times L^2(\mathbb{R}^3)$
 $\to H^1(\mathbb{R}^3) \times L^2(\mathbb{R}^3)$ given   by
 \begin{equation}
 \label{rds}
 \Phi(t, \omega, (u_0, z_0))
 = ( u(t,0,\omega, u_0), z(t,0,\omega, z_0)
  ) = (u(t,0,\omega, u_0), v(t,0, \omega, v_0) + h\omega(t)),
  \end{equation}
  for every
  $(t, \omega, (u_0, z_0)) \in \mathbb{R}^+ \times \Omega
  \times H^1(\mathbb{R}^3) \times L^2(\mathbb{R}^3)$,
where $v_0   = z_0 - h \omega (\tau)$.
  Then $\Phi$ is a continuous random dynamical system
  over
$(\Omega, \mathcal{F}, P, (\theta_t)_{t\in\mathbb{R}})$
on $\hone \times \ltwo$.
It is easy to verify that  $\Phi$  satisfies the
following  identity,
 for $P$-a.e.  $\omega \in \Omega$ and $t\ge 0$,
 \begin{equation}
 \label{shift}
 \Phi(t, \theta_{-t} \omega, (u_0, z_0))
 =
 (u(t , 0  , \theta_{-t} \omega, u_0), z(t , 0 ,  \theta_{-t}\omega, z_0))
 =
 (u(0, -t ,   \omega,u_0), z(0, -t ,  \omega, z_0)).
 \end{equation}

Throughout this paper, we always denote
 by $\mathcal{D}$ the collection of all
tempered random subsets of
$H^1(\mathbb{R}^3) \times L^2(\mathbb{R}^3) $, and will
prove  $\Phi$ has a $\mathcal{D}$-random attractor.

\section{Uniform  Estimates}
\setcounter{equation}{0}

In this section, we
 derive uniform estimates on  solutions of
problem \eqref{pde1}-\eqref{pde3}. These
   estimates  are  needed
 for proving  the existence of random absorbing sets
 and the pullback  asymptotic compactness of
  the random dynamical system $\Phi$.

Let $\delta>0 $ be small enough such that
\be
\label{delta}
\alpha -\delta>0 , \quad   \lambda +\delta^2 -\alpha \delta >0,
\ee
and denote by
\begin{equation}
\label{kappa}
\sigma ={\frac 12} \min \{\alpha-\delta, \delta, \delta c_2\},
  \end{equation}
  where $c_2$ is the positive constant in \eqref{f2}.

\begin{lem}
\label{lem31}
 Assume that $g \in L^2(\mathbb{R}^3)$, $ h  \in  H^1(\mathbb{R}^3)$ and
 \eqref{f1}-\eqref{f3} hold. Let
 $B=\{B(\omega)\}_{\omega \in \Omega}\in \mathcal{D}$.
  Then for $P$-a.e. $\omega \in \Omega$,
 there is $T= T (B, \omega) < 0$ such that for all $\tau \le T $,
 the solution $(u(\cdot,\tau, \omega, u_0), v(\cdot,\tau, \omega, v_0))$
 of problem \eqref{pde1}-\eqref{pde3}
 with $(u_0, v_0) \in B(\theta_\tau \omega)$ satisfies,
 for every $t \in [\tau, 0]$,
\begin{equation}
\label{lem31_1}
 \|  u(t, \tau, \omega, u_0)\|^2_{H^1(\mathbb{R}^3)}
   + \| v (t,\tau,   \omega, v_0 ) \|^2
 \le e^{-\sigma t} R (\omega),
 \end{equation}
 and
 \begin{equation}
\label{lem31_2}
 \int_\tau^t e^{\sigma \xi}
 \left ( \|  u(\xi, \tau, \omega, u_0)\|^2_{H^1(\mathbb{R}^3)}
  + \| v (\xi,\tau,   \omega, v_0  )\|^2  d\xi
  \right )
 \le R (\omega),
 \end{equation}
 where $R (\omega)$ is a positive  tempered random function.
 \end{lem}

\begin{proof}
 Taking the inner product of \eqref{pde2}
 with $v$ in $L^2(\mathbb{R}^3)$, we get
 $$
  {\frac 12} {\frac d{dt}} \| v\|^2
  + (\alpha - \delta) \| v \|^2
  + (\lambda + \delta^2 -\alpha \delta ) (u,v)
  -(\Delta u, v)
  +  ( f(x,u),   v)
  $$
  \begin{equation}
  \label{p31_1}
  =(g,v) + (\delta -\alpha)
  (h, v)\omega(t).
 \end{equation}
By \eqref{pde1} we have
\be
\label{p311_1}
(u,v) = {\frac 12} {\frac d{dt}} \| u \|^2
+ \delta \| u \|^2
- (u, h) \omega (t),
\ee
\be
\label{p311_2}
-(\Delta  u, v)
={\frac 12} {\frac d{dt}} \| \nabla u \|^2
+ \delta \| \nabla u \|^2
-(\nabla u, \nabla h) \omega (t),
\ee
and
\be
\label{p311_3}
(f(x,u), v )
= {\frac d{dt}} \int_{\R^3} F(x, u) dx
+ \delta (f(x,u), u)
- (f(x,u), h) \omega (t).
\ee
It follows from \eqref{p31_1}-\eqref{p311_3} that
 $$
   {\frac d{dt}}  \left (
  \| v \|^2 + (\lambda +\delta^2 -\alpha \delta) \| u \|^2
  + \| \nabla u \|^2 + 2 \int_{\mathbb{R}^3} F(x, u)dx
 \right )
 $$
 $$
  +2 (\alpha -\delta) \| v\|^2
  + 2 \delta (\lambda + \delta^2 -\alpha \delta ) \| u \|^2
  +2 \delta \| \nabla u \|^2
  + 2 \delta (f(x,u), u)
 $$
 $$
 = 2(\lambda + \delta^2 -\alpha \delta) (h,u) \omega(t)
 + 2 (\nabla u, \nabla h) \omega (t)
 + 2 (f(x, u), h) \omega(t)
 $$
 \begin{equation}
 \label{p31_2}
 + 2(g,v) + 2 (\delta -\alpha) (h,v) \omega (t).
 \end{equation}
 We  now estimate every term on the right-hand side
 of \eqref{p31_2}.
For the first term, by \eqref{delta}  we have
\be
\label{p311_10}
2(\lambda + \delta^2 -\alpha \delta) (h,u) \omega(t)
\le
(\lambda + \delta^2 -\alpha \delta) \| u \|^2
+ c \| h\|^2 |\omega (t) |^2.
\ee
The second term on the right-hand side of \eqref{p31_2}
satisfies
\be
\label{p311_11}
2 (\nabla u, \nabla h) \omega (t)
\le \delta \| \nabla u \|^2
+ c \| \nabla h \|^2 | \omega (t) |^2.
\ee
For the  third  term on the right-hand side of \eqref{p31_2},
 by  \eqref{f1} and \eqref{F2}, we obtain
 $$
 2 (f(x,u), h) \omega (t)
 \le 2 \| \phi_1 \| \|h\| |\omega (t) |
 + c \left ( \int_{\mathbb{R}^3} |u|^{\gamma +1}
 \right )^{\frac \gamma{\gamma +1}} \| h\|_{\gamma +1} |\omega (t) |
 $$
 $$
 \le
 2 \| \phi_1 \| \|h\| |\omega (t) |
 +
 c \left ( \int_{\mathbb{R}^3} (F(x,u) + \phi_3)
 \right )^{\frac \gamma{\gamma +1}} \| h\|_{\gamma +1} |\omega (t) |
 $$
\begin{equation}
\label{p31_5}
 \le
 2 \| \phi_1 \| \|h\| |\omega (t) |
 +
 \delta c_2
  \int_{\mathbb{R}^3}  F(x,u)dx
   + \delta c_2
  \int_{\mathbb{R}^3}
  \phi_3(x) dx + c
   \| h\|_{H^1}^{\gamma +1} |\omega (t) |^{\gamma +1}.
\end{equation}
Similarly, by Young's inequality, the last two terms
  on the right-hand side of \eqref{p31_2} are bounded by
\be
\label{p311_12}
2 |(g,v)| + 2 |(\delta -\alpha) (h,v) \omega (t)|
\le
  (\alpha -\delta ) \| v \|^2
+ c \| h\|^2 |\omega (t) |^2 + c \| g \|^2.
\ee
 By \eqref{f2} we also have
 \begin{equation}
 \label{p31_3}
 (f(x,u), u)
 \ge c_2 \int_{\mathbb{R}^3} F(x, u) dx
 + \int_{\mathbb{R}^3} \phi_2 (x) dx.
 \end{equation}
By
\eqref{p31_2}-\eqref{p31_3}, we find that
$$
 {\frac d{dt}} \left (
  \| v \|^2 + (\lambda +\delta^2 -\alpha \delta) \| u \|^2
  + \| \nabla u \|^2 + 2 \int_{\mathbb{R}^3} F(x, u)dx
 \right )
 $$
 $$
  +  (\alpha -\delta) \| v\|^2
  +   \delta (\lambda + \delta^2 -\alpha \delta ) \| u \|^2
  +  \delta \| \nabla u \|^2
  +   \delta c_2 \int_{\mathbb{R}^3} F(x, u) dx
 $$
\begin{equation}
\label{p31_10}
\le c \left (
1 + |\omega (t) |^2 + |\omega (t)|^{\gamma +1} \right ).
\end{equation}
By \eqref{F2} and \eqref{kappa}  we have
$$
\delta c_2 \int_{\mathbb{R}^3} F(x, u) dx
\ge 2 \sigma \int_{\mathbb{R}^3} F(x, u) dx
+ (2\sigma -\delta c_2 ) \int_{\mathbb{R}^3} \phi_3(x) dx,
$$
which along with
\eqref{p31_10} implies  that
$$
 {\frac d{dt}} \left (
  \| v \|^2 + (\lambda +\delta^2 -\alpha \delta) \| u \|^2
  + \| \nabla u \|^2 + 2 \int_{\mathbb{R}^3} F(x, u)dx
 \right )
 $$
 $$
  +  \sigma \left ( \| v\|^2
  +    (\lambda + \delta^2 -\alpha \delta ) \| u \|^2
  +   \| \nabla u \|^2
  +  2 \int_{\mathbb{R}^3} F(x, u) dx \right )
 $$
  \begin{equation}
\label{p31_11}
 + \sigma \left ( \| v\|^2
  +    (\lambda + \delta^2 -\alpha \delta ) \| u \|^2
  +   \| \nabla u \|^2 \right )
\le c \left (
1 + |\omega (t) |^2 + |\omega (t)|^{\gamma +1} \right ).
\end{equation}
Integrating \eqref{p31_11} on $(\tau, t)$
with $t \le 0$, we get
$$
 e^{\sigma t} \left (
 \| v(t, \tau, \omega)\|^2
+ (\lambda +\delta^2 -\alpha \delta) \|u(t, \tau, \omega)\|^2
+ \| \nabla u(t, \tau, \omega) \|^2
+ 2 \int_{\mathbb{R}^3} F(x, u) dx
\right )
$$
$$
+ \sigma \int_\tau^t e^{\sigma \xi}
\left ( \| v\|^2
  +    (\lambda + \delta^2 -\alpha \delta ) \| u \|^2
  +   \| \nabla u \|^2 \right )d\xi
  $$
  $$
  \le
  e^{\sigma \tau}
  \left ( \| v_0\|^2
  +    (\lambda + \delta^2 -\alpha \delta ) \| u_0 \|^2
  +   \| \nabla u_0 \|^2  + 2 \int_{\mathbb{R}^3} F(x,u_0) dx
  \right )
  $$
  \begin{equation}
  \label{p31_20}
  +c \int_\tau ^t
  e^{\sigma \xi}
  \left (
1 + |\omega (\xi) |^2 + |\omega (\xi)|^{\gamma +1} \right )d\xi.
 \end{equation}
 By \eqref{F3} we have
 $$
 \int_{\mathbb{R}^3} F(x,u_0) dx
 \le c \left (
 1 + \|u_0\|^2 + \| u_0\|_{H^1}^{\gamma +1}
 \right ),
 $$
 which along with   $(u_0, v_0) \in B(\theta_\tau \omega)$  implies that
 $$
  e^{\sigma \tau}
  \left ( \| v_0\|^2
  +    (\lambda + \delta^2 -\alpha \delta ) \| u_0 \|^2
  +   \| \nabla u_0 \|^2  + 2 \int_{\mathbb{R}^3} F(x,u_0) dx
  \right )
  $$
  \begin{equation}
  \label{p31_300}
  \le c e^{\sigma \tau}
  \left ( 1+  \|v_0\|^2
  + \| u_0 \|_{H^1}^2
  + \| u_0 \|_{H^1}^{\gamma +1}
  \right )
  \to 0 \quad \mbox{as} \ \tau \to -\infty.
  \end{equation}
  Therefore, there exists $T=T(B, \omega)<0$
  such that for all $\tau \le T$,
   \begin{equation}
  \label{p31_30}
  e^{\sigma \tau}
  \left ( \| v_0\|^2
  +    (\lambda + \delta^2 -\alpha \delta ) \| u_0 \|^2
  +   \| \nabla u_0 \|^2  + 2 \int_{\mathbb{R}^3} F(x,u_0) dx
  \right ) \le r(\omega),
   \end{equation}
   where
   $$r(\omega) =
   \int_{-\infty} ^0
  e^{\sigma \xi}
  \left (
1 + |\omega (\xi) |^2 + |\omega (\xi)|^{\gamma +1} \right )d\xi.
$$
Notice that $r(\omega)$ is well defined since $\omega(\xi)$ has
at most linear growth rate as $|\xi| \to \infty$.
By   and \eqref{p31_20} and \eqref{p31_30}
we obtain that,  for all $\tau \le T$ and $t\in [\tau, 0]$,
  $$
 e^{\sigma t}
 \left (
 \| v(t, \tau, \omega)\|^2
+ (\lambda +\delta^2 -\alpha \delta) \|u(t, \tau, \omega)\|^2
+ \| \nabla u(t, \tau, \omega) \|^2 + 2 \int_{\R^3} F(x,u) dx
\right )
$$
\be
\label{p31_40}
+  \int_\tau^t e^{\sigma \xi}
\left ( \| v\|^2
  +    (\lambda + \delta^2 -\alpha \delta ) \| u \|^2
  +   \| \nabla u \|^2 \right )d\xi
  \le c (1 + r(\omega)).
 \ee
By \eqref{F2}, we find that, for all $t \le 0$,
\be
\label{p31_42}
-2 e^{\sigma t} \int_{\R^3} F(x,u) dx
\le
 2 e^{\sigma t} \int_{\R^3} \phi_3(x) dx
\le 2 \int_{\R^3} | \phi_3 (x) | dx.
 \ee
By
\eqref{p31_40} and \eqref{p31_42} we have that,
for all $\tau \le T$ and $t\in [\tau, 0]$,
  $$
 e^{\sigma t}
 \left (
 \| v(t, \tau, \omega)\|^2
+ (\lambda +\delta^2 -\alpha \delta) \|u(t, \tau, \omega)\|^2
+ \| \nabla u(t, \tau, \omega) \|^2
\right )
$$
$$
+  \int_\tau^t e^{\sigma \xi}
\left ( \| v\|^2
  +    (\lambda + \delta^2 -\alpha \delta ) \| u \|^2
  +   \| \nabla u \|^2 \right )d\xi
  \le c (1 + r(\omega)),
 $$
which implies
  \eqref{lem31_1} and
  \eqref{lem31_2}   with
  $R(\omega) = c(1 + r(\omega))$.
Next we show that $R(\omega)$ is tempered, that is,
for every $\beta>0$, we want to prove
\be
\label{p311_80}
e^{\beta \tau} R(\theta_\tau \omega) \to 0
\quad \mbox{as} \quad \tau \to -\infty.
\ee
Without loss of generality, we now assume $\beta \le
\sigma$. Then we have
$$
e^{\beta \tau} R(\theta_\tau \omega)
=c e^{\beta \tau}
+ ce^{\beta \tau}\int^0_{-\infty}
e^{\sigma \xi} \left ( |(\theta_\tau \omega) (\xi )|^2
+ |(\theta_\tau \omega) (\xi )|^{\gamma +1}
\right ) d \xi
$$
$$
\le
 c e^{\beta \tau}
+
ce^{\beta \tau}\int^0_{-\infty}
e^{\beta \xi} \left ( |(\theta_\tau \omega) (\xi )|^2
+ |(\theta_\tau \omega) (\xi )|^{\gamma +1}
\right ) d \xi
$$
$$
\le
 c e^{\beta \tau}
+
ce^{\beta \tau}\int^0_{-\infty}
e^{\beta \xi} \left ( | \omega (\tau )|^2 + | \omega (\tau )|^{\gamma +1}
\right ) d\xi
$$
$$ +
ce^{\beta \tau}\int^0_{-\infty}
e^{\beta \xi} \left ( | \omega (\tau + \xi )|^2 + | \omega (\tau + \xi )|^{\gamma +1}
\right ) d\xi
$$
\be
\label{p311_81}
\le
 c e^{\beta \tau}
+ {\frac c\beta} e^{\beta \tau}
  \left ( | \omega (\tau )|^2 + | \omega (\tau )|^{\gamma +1} \right )
  +
c  \int^\tau_{-\infty}
e^{\beta s} \left ( | \omega (s )|^2 + | \omega (s )|^{\gamma +1}
\right ) ds.
\ee
Then \eqref{p311_80} follows from
\eqref{p311_81} since
$\omega$ has at most linear growth rate at infinity.
This completes the proof.
\end{proof}

 We now derive an energy equation for problem \eqref{pde1}-\eqref{pde3}.
To this end,
denote by, for $(u,v) \in \hone \times \ltwo$,
\be
\label{ener1}
E(u,v) = \| v \|^2
+ (\lambda + \delta^2 -\alpha \delta) \| u \|^2
+ \| \nabla u \|^2
+ 2 \int_{\R^3} F(x, u) dx,
\ee
and
$$
\Psi(u(t, \tau, \omega, u_0),v(t, \tau, \omega, v_0))
 $$
$$=  -2(\alpha-\delta -2 \sigma ) \| v \|^2
-2 (\delta -2\sigma)  (\lambda + \delta^2 -\alpha \delta) \| u \|^2
-2 (\delta -2 \sigma) \| \nabla u \|^2
$$
$$
+ 8 \sigma \int_{\R^3} F(x, u) dx
-2 \delta \int_{\R^3} f(x,u)u dx
+ 2(\lambda +\delta^2 -\alpha \delta) (u, h) \omega(t)
$$
\be
\label{ener2}
+ 2(\nabla u, \nabla h)\omega(t)
+2\omega(t) \int_{\R^3} f(x, u) h(x)  dx
+ 2(g,v) + 2(\delta -\alpha) (v, h) \omega(t).
\ee
Then it follows from \eqref{p31_2} that
\be
\label{ener3}
{\frac d{dt}} E
+ 4 \sigma E
=\Psi.
\ee
Integrating \eqref{ener3} on $(\tau,  t)$ we get
\be
\label{ener}
E(u(t, \tau, \omega, u_0), v(t, \tau, \omega, v_0) )
$$
$$
=
e^{-4 \sigma (t-\tau)} E(u_0, v_0)
+ \int_\tau^t e^{4 \sigma (\xi -t)}
\Psi(u(\xi, \tau, \omega, u_0), v(\xi, \tau, \omega, v_0))d\xi.
\ee
The energy equation \eqref{ener} will be used
to prove the pullback asymptotic compactness
of solutions in the last section.

In what follows,  we   derive   uniform estimates
on the tails of solutions when $x$ and $t$ approach infinity.
These estimates will be used to overcome the difficulty caused
by   non-compactness of embeddings $\hone \hookrightarrow
L^p(\R^3)$ for $ p \le 6$, and
  are crucial for proving the pullback asymptotic compactness
of the random dynamical system.
Given $k\ge 1$, denote by   $Q_{k} =$\{$
 x\in \mathbb{R}^3  $: $|x|<k\}$
 and $\mathbb{R}^3 \backslash Q_{k}$   the complement of $Q_{k}$.

\begin{lem}
\label{lem32}
 Assume that $g \in L^2(\mathbb{R}^3)$, $ h  \in  H^1(\mathbb{R}^3)$ and
 \eqref{f1}-\eqref{f3} hold. Let
 $B=\{B(\omega)\}_{\omega \in \Omega}\in \mathcal{D}$.
  Then for every $\epsilon>0$ and
   $P$-a.e. $\omega \in \Omega$,
 there exist $T= T (B, \omega, \epsilon) < 0$
  and $k_0=k_0(\omega, \epsilon)>0$
  such that for all $\tau \le T $ and $k \ge k_0$,
 the solution $(u(\cdot,\tau, \omega, u_0), v(\cdot,\tau, \omega, v_0))$
 of problem \eqref{pde1}-\eqref{pde3}
 with $(u_0, v_0) \in B(\theta_\tau \omega)$ satisfies,
 for any $t \in [\tau, 0]$,
\begin{equation}
\label{lem32_1}
\int_{\mathbb{R}^3 \backslash Q_{k}}
\left (
 |  u(t, \tau, \omega, u_0) |^2  + |\nabla u(t, \tau, \omega, u_0)|^2
   +  | v (t,\tau,   \omega, v_0 )  |^2
   \right ) dx
 \le \epsilon e^{-\sigma t}.
 \end{equation}
 \end{lem}

 \begin{proof}
  Take a smooth function $\rho$ such that $0 \leq \rho \leq 1$   for all  $s \in \mathbb{R}$ and
\begin{equation}
\label{rho}
\rho (s)  =
\left \{
\begin{array}{ll}
    0, & \quad  \mbox{if }  \quad   |s| < 1,\\
  1, & \quad  \mbox{if }   \quad |s| > 2.
\end{array}
\right.
\end{equation}
Then there is a positive constant $c$ such that
$|\rho^\prime (s)|   \le c$
for all $s\in \mathbb{R}$.

Taking the inner product of \eqref{pde2}
 with $\rho\left ({\frac {|x|^2}{k^2}}  \right ) v$ in $L^2(\mathbb{R}^3)$, we get
 $$
  {\frac 12} {\frac d{dt}}
 \int_{\mathbb{R}^3}
  \rho\left ({\frac {|x|^2}{k^2}}  \right )|v|^2 dx
  + (\alpha - \delta)\int_{\mathbb{R}^3}
  \rho\left ({\frac {|x|^2}{k^2}}  \right ) |v|^2 dx
  $$
  $$
  + (\lambda + \delta^2 -\alpha \delta )\int_{\mathbb{R}^3}
  \rho\left ({\frac {|x|^2}{k^2}}  \right ) uv dx
  -\int_{\mathbb{R}^3}
  \rho\left ({\frac {|x|^2}{k^2}}  \right )v
   \Delta u  dx
   $$
  \begin{equation}
  \label{p32_1}
  + \int_{\mathbb{R}^3}
  \rho\left ({\frac {|x|^2}{k^2}}  \right )  f(x,u) v dx
  = \int_{\mathbb{R}^3}
  \rho\left ({\frac {|x|^2}{k^2}}  \right )
  ( g v
   + (\delta -\alpha)
   vh\omega(t) ) dx.
 \end{equation}
 By \eqref{pde1} we find that
$$
\int_{\mathbb{R}^3}
  \rho\left ({\frac {|x|^2}{k^2}}  \right ) uv dx
={\frac 12} {\frac d{dt}}
\int_{\mathbb{R}^3}
  \rho\left ({\frac {|x|^2}{k^2}}  \right ) |u|^2 dx
$$
\be
\label{p32_2_a1}
+ \delta
\int_{\mathbb{R}^3}
  \rho\left ({\frac {|x|^2}{k^2}}  \right ) |u|^2 dx
- \int_{\mathbb{R}^3}
  \rho\left ({\frac {|x|^2}{k^2}}  \right ) u h \omega (t) dx ,
\ee
 $$
 -\int_{\mathbb{R}^3}
  \rho\left ({\frac {|x|^2}{k^2}}  \right )v
   \Delta u  dx
   = \int_{\mathbb{R}^3} \nabla u {\frac {2x}{k^2}}
  \rho^\prime \left ({\frac {|x|^2}{k^2}}  \right ) vdx
  + {\frac 12} {\frac d{dt}}
   \int_{\mathbb{R}^3}
  \rho\left ({\frac {|x|^2}{k^2}}  \right )|\nabla u|^2 dx
   $$
   \begin{equation}
   \label{p32_2}
  + \delta \int_{\mathbb{R}^3}
  \rho\left ({\frac {|x|^2}{k^2}}  \right )|\nabla u|^2dx
  -\int_{\mathbb{R}^3}
  \rho\left ({\frac {|x|^2}{k^2}}  \right )
  \nabla u \nabla h \omega (t) dx,
  \end{equation}
and
$$
 \int_{\mathbb{R}^3}
  \rho\left ({\frac {|x|^2}{k^2}}  \right )
f(x, u) v dx
={\frac {d}{dt}}
\int_{\mathbb{R}^3}
  \rho\left ({\frac {|x|^2}{k^2}}  \right ) F(x,u) dx
$$
\be
\label{p32_2_a2}
+ \delta
\int_{\mathbb{R}^3}
  \rho\left ({\frac {|x|^2}{k^2}}  \right )
f(x, u) u dx
- \int_{\mathbb{R}^3}
  \rho\left ({\frac {|x|^2}{k^2}}  \right )
f(x, u) h \omega (t)  dx.
\ee
It follows from
 \eqref{p32_1}-\eqref{p32_2_a2}   that
 $$
{\frac d{dt}}   \int_{\mathbb{R}^3}
  \rho\left ({\frac {|x|^2}{k^2}}  \right )
  \left (|v|^2 +
  (\lambda +\delta^2 -\alpha \delta)  | u  |^2
  +  | \nabla u  |^2 + 2   F(x, u)
 \right ) dx
 $$
 $$
  +\int_{\mathbb{R}^3}
  \rho\left ({\frac {|x|^2}{k^2}}  \right )
  \left (
  2 (\alpha -\delta)  | v |^2
  + 2 \delta (\lambda + \delta^2 -\alpha \delta )  | u  |^2
  +2 \delta  | \nabla u  |^2
  + 2 \delta  f(x,u)  u  \right ) dx
 $$
 $$
 = 2(\lambda + \delta^2 -\alpha \delta) \int_{\mathbb{R}^3}
  \rho\left ({\frac {|x|^2}{k^2}}  \right )
 hu \omega(t)dx
 - 4 \int_{\mathbb{R}^3}
  \rho^\prime\left ({\frac {|x|^2}{k^2}}  \right )
  v \nabla u {\frac x{k^2}} dx
  $$
  $$
  + 2 \int_{\mathbb{R}^3}
  \rho\left ({\frac {|x|^2}{k^2}}  \right )
  f(x,u) h\omega(t) dx
  + 2\int_{\mathbb{R}^3}
  \rho\left ({\frac {|x|^2}{k^2}}  \right )
       \nabla u  \nabla h  \omega (t) dx
 $$
 \begin{equation}
 \label{p32_3}
   + 2 \int_{\mathbb{R}^3}
  \rho\left ({\frac {|x|^2}{k^2}}  \right )
 \left (  gv +   (\delta -\alpha) hv \omega (t)
 \right ) dx.
 \end{equation}
 By \eqref{f2} we have
 \begin{equation}
 \label{p32_4}
 \int_{\mathbb{R}^3}
  \rho\left ({\frac {|x|^2}{k^2}}  \right )
  f(x,u)  u dx
 \ge c_2  \int_{\mathbb{R}^3}
  \rho\left ({\frac {|x|^2}{k^2}}  \right )
   F(x, u) dx
 +  \int_{\mathbb{R}^3}
  \rho\left ({\frac {|x|^2}{k^2}}  \right )
   \phi_2 (x) dx.
 \end{equation}
 By \eqref{f1} and \eqref{F2}
 as in \eqref{p31_5}, we also have
 $$
 2\int_{\mathbb{R}^3}
  \rho\left ({\frac {|x|^2}{k^2}}  \right )
    f(x,u)  h  \omega (t) dx
    \le \int_{\mathbb{R}^3}
  \rho\left ({\frac {|x|^2}{k^2}}  \right )
  |\phi_1|^2 dx
  + c \int_{\mathbb{R}^3}
  \rho\left ({\frac {|x|^2}{k^2}}  \right )
  |h|^2 |\omega(t)|^2 dx
  $$
\begin{equation}
\label{p32_10}
 +
 \delta c_2
   \int_{\mathbb{R}^3}
  \rho\left ({\frac {|x|^2}{k^2}}  \right )
  \left (  F(x,u) +
  \phi_3(x) \right ) dx
  + c \int_{\mathbb{R}^3}
  \rho\left ({\frac {|x|^2}{k^2}}  \right )
  |h|^{\gamma +1} |\omega (t)|^{\gamma +1} dx.
\end{equation}
By the definition of $\rho$ in  \eqref{rho} we have
\begin{equation}
\label{p32_11}
   \int_{\mathbb{R}^3}
  |\rho^\prime\left ({\frac {|x|^2}{k^2}}  \right )
  v \nabla u {\frac x{k^2}}| dx
  \le
    \int_{k \le |x| \le \sqrt{2} k}
  |\rho^\prime| |v| |\nabla u | {\frac {|x|}{k^2}} dx
  \le {\frac ck} \left (\|\nabla u \|^2 + \| v\|^2
  \right ).
 \end{equation}
Using  Young's   inequality
to estimate the remaining terms on
  on the right-hand side of
\eqref{p32_3}, by
\eqref{p32_4}-\eqref{p32_11}, we find that
$$
{\frac d{dt}}   \int_{\mathbb{R}^3}
  \rho\left ({\frac {|x|^2}{k^2}}  \right )
  \left (|v|^2 +
  (\lambda +\delta^2 -\alpha \delta)  | u  |^2
  +  | \nabla u  |^2 + 2   F(x, u)
 \right ) dx
 $$
  $$
  +\int_{\mathbb{R}^3}
  \rho\left ({\frac {|x|^2}{k^2}}  \right )
  \left (
    (\alpha -\delta)  | v |^2
  +   \delta (\lambda + \delta^2 -\alpha \delta )  | u  |^2
  +  \delta  | \nabla u  |^2
  +   \delta c_2  F(x,u)    \right ) dx
 $$
  $$
 \le
 {\frac ck} (\| \nabla u \|^2 + \| v \|^2)
 + c |\omega (t)|^2 \int_{\mathbb{R}^3}
  \rho\left ({\frac {|x|^2}{k^2}}  \right )
  (|h|^2 + |\nabla h |^2 ) dx
  $$
 \begin{equation}
 \label{p32_20}
  + c
  \int_{\mathbb{R}^3}
  \rho\left ({\frac {|x|^2}{k^2}}  \right )
  ( |\phi_1|^2  + |\phi_2| + |\phi_3|
  + |g|^2 + |\omega (t)|^{\gamma +1}
    |h|^{\gamma +1} ) dx.
 \end{equation}
For the last two terms on the right-hand side of
 \eqref{p32_20}, we find that
 there exists $k_1 =k_1(\epsilon)\ge 1$ such that
 for all $k \ge k_1$,
$$
c |\omega (t)|^2 \int_{\mathbb{R}^3}
  \rho\left ({\frac {|x|^2}{k^2}}  \right )
  (|h|^2 + |\nabla h |^2 ) dx
$$
$$
  + c
  \int_{\mathbb{R}^3}
  \rho\left ({\frac {|x|^2}{k^2}}  \right )
  ( |\phi_1|^2  + |\phi_2| + |\phi_3|
  + |g|^2 + |\omega (t)|^{\gamma +1}
    |h|^{\gamma +1} ) dx
$$
$$
= c |\omega (t)|^2 \int_{|x| \ge k}
  \rho\left ({\frac {|x|^2}{k^2}}  \right )
  (|h|^2 + |\nabla h |^2 ) dx
$$
$$
  + c
  \int_{|x| \ge k}
  \rho\left ({\frac {|x|^2}{k^2}}  \right )
  ( |\phi_1|^2  + |\phi_2| + |\phi_3|
  + |g|^2 + |\omega (t)|^{\gamma +1}
    |h|^{\gamma +1} ) dx
$$
$$
 \le  c |\omega (t)|^2 \int_{|x| \ge k}
  (|h|^2 + |\nabla h |^2 ) dx
$$
$$
  + c
  \int_{|x| \ge k}
  ( |\phi_1|^2  + |\phi_2| + |\phi_3|
  + |g|^2 + |\omega (t)|^{\gamma +1}
    |h|^{\gamma +1} ) dx
$$
\be
\label{p32_30_a1}
\le
  c\epsilon (1 + |\omega (t)|^2 + |\omega (t) |^{\gamma +1} ),
\ee
where we have used the fact that
  $\phi_1, g \in L^2(\mathbb{R}^n)$,
 $\phi_2, \phi_3 \in L^1(\mathbb{R}^n)$,
 $h \in H^1(\mathbb{R}^n)$,
and the embedding $\hone \hookrightarrow L^{\gamma +1}(\R^3)$
with $\gamma \le 3$.
It follows from \eqref{p32_20}-\eqref{p32_30_a1} that,
   for all $k\ge k_1$,
 $$
{\frac d{dt}}   \int_{\mathbb{R}^3}
  \rho\left ({\frac {|x|^2}{k^2}}  \right )
  \left (|v|^2 +
  (\lambda +\delta^2 -\alpha \delta)  | u  |^2
  +  | \nabla u  |^2 + 2   F(x, u)
 \right ) dx
 $$
  $$
  +\int_{\mathbb{R}^3}
  \rho\left ({\frac {|x|^2}{k^2}}  \right )
  \left (
    (\alpha -\delta)  | v |^2
  +   \delta (\lambda + \delta^2 -\alpha \delta )  | u  |^2
  +  \delta  | \nabla u  |^2
  +   \delta c_2  F(x,u)    \right ) dx
 $$
  \begin{equation}
  \label{p32_30}
 \le
 {\frac ck} (\| \nabla u \|^2 + \| v \|^2)
  + c\epsilon (1 + |\omega (t)|^2 + |\omega (t) |^{ \gamma +1} ).
 \end{equation}
 By \eqref{F2}, \eqref{kappa} and \eqref{p32_30} we find that
 for all $k\ge k_1$,
 $$
{\frac d{dt}}   \int_{\mathbb{R}^3}
  \rho\left ({\frac {|x|^2}{k^2}}  \right )
  \left (|v|^2 +
  (\lambda +\delta^2 -\alpha \delta)  | u  |^2
  +  | \nabla u  |^2 + 2   F(x, u)
 \right ) dx
 $$
  $$
  +\sigma
   \int_{\mathbb{R}^3}
  \rho\left ({\frac {|x|^2}{k^2}}  \right )
  \left (|v|^2 +
  (\lambda +\delta^2 -\alpha \delta)  | u  |^2
  +  | \nabla u  |^2 + 2   F(x, u)
 \right ) dx
  $$
  \begin{equation}
  \label{p32_40}
 \le
 {\frac ck} (\| \nabla u \|^2 + \| v \|^2)
  + c\epsilon (1 + |\omega (t)|^2 + |\omega (t) |^{\gamma +1} ).
 \end{equation}
 Integrating \eqref{p32_40} on $(\tau, t)$
 with $t \le 0$,
 by Lemma \ref{lem31} we find that, for
 all $k\ge k_1$,
 $$
 e^{\sigma t}
  \int_{\mathbb{R}^3}
  \rho ( {\frac {|x|^2}{k^2}}    )
  \left (|v(t, \tau, \omega)|^2 +
  (\lambda +\delta^2 -\alpha \delta)  | u(t,\tau,\omega)  |^2
  +  | \nabla u (t,\tau,\omega) |^2 + 2   F(x, u)
 \right )
 $$
 $$
 \le
 e^{\sigma \tau}
  \int_{\mathbb{R}^3}
  \rho  ({\frac {|x|^2}{k^2}}    )
  \left (|v_0|^2 +
  (\lambda +\delta^2 -\alpha \delta)  | u_0  |^2
  +  | \nabla u_0  |^2 + 2   F(x, u_0)
 \right ) dx
 $$
 $$
 + {\frac ck}\int_\tau^t e^{\sigma \xi}
 (\| \nabla u(\xi) \|^2 + \| v(\xi) \|^2 ) d\xi
 + c\epsilon \int_\tau^t e^{\sigma \xi}
 (|\omega (\xi)|^2 + |\omega (\xi)|^{\gamma +1}) d\xi
 + c\epsilon
 $$
  $$
 \le
 e^{\sigma \tau}
  \int_{\mathbb{R}^3}
  \rho  ({\frac {|x|^2}{k^2}}    )
  \left (|v_0|^2 +
  (\lambda +\delta^2 -\alpha \delta)  | u_0  |^2
  +  | \nabla u_0  |^2 + 2   F(x, u_0)
 \right ) dx
 $$
\begin{equation}
\label{p32_90}
 + {\frac ck}R(\omega)
 + c\epsilon \int_{-\infty}^0 e^{\sigma \xi}
 (|\omega (\xi)|^2 + |\omega (\xi)|^{\gamma +1}) d\xi
 + c\epsilon,
 \end{equation}
where $R(\omega)$ is the positive tempered random function
in Lemma \ref{lem31}.
 As in \eqref{p31_300},
 the first term on the right-hand side of \eqref{p32_90}
 goes to zero as $\tau \to -\infty$. Hence,
 there exists $T=T(B, \omega, \epsilon)<0$ such that all
 $\tau \le T$,
 \begin{equation}
 \label{p32_91}
 e^{\sigma \tau}
  \int_{\mathbb{R}^3}
  \rho  ({\frac {|x|^2}{k^2}}    )
  \left (|v_0|^2 +
  (\lambda +\delta^2 -\alpha \delta)  | u_0  |^2
  +  | \nabla u_0  |^2 + 2   F(x, u_0)
 \right ) dx \le \epsilon.
 \end{equation}
 By \eqref{p32_90}-\eqref{p32_91},
 there exists $  k_2(\epsilon) \ge k_1(\epsilon)$
 such that for all $\tau \le T$ and $k \ge k_2$,
\be
\label{p32_95}
   e^{\sigma t}
   \int_{\mathbb{R}^3}
  \rho ( {\frac {|x|^2}{k^2}}    )
  \left (|v(t, \tau, \omega)|^2 +
  (\lambda +\delta^2 -\alpha \delta)  | u(t,\tau,\omega)  |^2
  +  | \nabla u(t, \tau, \omega)  |^2 + 2   F(x, u)
 \right )dx \le   \epsilon c \ r(\omega),
 \ee
 where $r(\omega) =
   1 + R(\omega) + \int_{-\infty}^0 e^{\sigma \xi}
 (|\omega (\xi)|^2 + |\omega (\xi)|^{\gamma +1}) d\xi  $.
By \eqref{F2} we have, for $t \le 0$,
$$-2 e^{\sigma t} \int_{\R^3}
\rho ( {\frac {|x|^2}{k^2}}    ) F(x, u) dx
\le 2 e^{\sigma t}
\int_{\R^3}
\rho ( {\frac {|x|^2}{k^2}}    ) \phi_3 (x) dx
\le
2
\int_{|x| \ge k }
\rho ( {\frac {|x|^2}{k^2}}    ) \phi_3 (x) dx
\le
2 \int_{|x| \ge k }
 | \phi_3 (x) | dx.
$$
Since $\phi_3 \in L^1(\R^3)$, there is $k_3 =k_3(\epsilon) \ge k_2$
such that for all $k \ge k_3$, the right-hand side of the above
is bounded by $\epsilon$. Hence we have, for all
$k \ge k_3$ and $t \le 0$,
\be
\label{p32_96}
-2 e^{\sigma t} \int_{\R^3}
\rho ( {\frac {|x|^2}{k^2}}    ) F(x, u) dx
 \le \epsilon.
\ee
By \eqref{p32_95}-\eqref{p32_96} we get that, for all
$\tau \le T$, $t \in [\tau, 0]$  and $k \ge k_3$,
$$
 e^{\sigma t}  \int_{\mathbb{R}^3}
  \left (|v(t, \tau, \omega)|^2 +
  (\lambda +\delta^2 -\alpha \delta)  | u(t,\tau,\omega)  |^2
  +  | \nabla u(t,\tau,\omega) |^2
 \right ) dx
   \le  \epsilon + c   \epsilon \  r(\omega).
 $$
By the definition of $\rho$ in \eqref{rho}, we finally obtain that,
 for all $\tau \le T$, $t \in [\tau, 0]$  and $k \ge k_3$,
 $$
  e^{\sigma t}  \int_{|x| \ge \sqrt{2} k}
  \left (|v(t, \tau, \omega)|^2 +
  (\lambda +\delta^2 -\alpha \delta)  | u(t,\tau,\omega)  |^2
  +  | \nabla u(t,\tau,\omega) |^2
 \right ) dx
 $$
  $$
\le
 e^{\sigma t}  \int_{\mathbb{R}^3}\rho ( {\frac {|x|^2}{k^2}}    )
  \left (|v(t, \tau, \omega)|^2 +
  (\lambda +\delta^2 -\alpha \delta)  | u(t,\tau,\omega)  |^2
  +  | \nabla u(t,\tau,\omega) |^2
 \right ) dx
   \le  \epsilon + c   \epsilon \  r(\omega).
 $$
 which completes the proof. \end{proof}

\section{Random Attractors}
\setcounter{equation}{0}

In this section, we prove   existence of a
$\mathcal{D}$-random attractor  for the
stochastic wave equation on $\R^3$.
We first show that the random dynamical
system $\Phi$ has a closed random absorbing
set in $\mathcal{D}$, and then prove that
$\Phi$ is $\mathcal{D}$-pullback asymptotically
 compact.

 By Lemma \ref{lem31} we find that
 for every   $B=\{B(\omega)\}_{\omega \in \Omega}
\in \mathcal{D}$, and  $P$-a.e. $\omega \in \Omega$,
there is $T=T(B, \omega)<0$ such that
for all $\tau  \le  T$, the solution
$(u,v)$ of problem \eqref{pde1}-\eqref{pde3}
with $(u_0, v_0) \in B(\theta_\tau \omega)$
satisfies
\be
\label{sec5_1}
\| u(0, \tau, \omega, u_0)\|^2_{\hone}
+ \| v(0, \tau, \omega, v_0)\|^2
\le R(\omega),
\ee
where $R(\omega)$ is the positive tempered random function
in Lemma \ref{lem31}.
Since
$z(t, \tau, \omega, z_0) =
v(t,\tau, \omega, v_0)
+ h \omega(t)$
with
$z_0 = v_0 + h \omega (\tau)$, it follows from
\eqref{sec5_1} that
$(u(t, \tau, \omega, u_0), z(t, \tau, \omega, z_0) )$
 with $(u_0, z_0) \in B(\theta_\tau \omega)$
satisfies,  for all $\tau \le T$,
$$
\| u(0, \tau, \omega, u_0)\|_{H^1}^2
+ \| z(0,\tau, \omega, z_0)\|^2
=
\| u(0, \tau, \omega, u_0)\|_{H^1}^2
+ \| v(0,\tau, \omega, v_0  ) \|^2
\le R(\omega),
$$
which along with \eqref{shift}
implies that, for all $t \ge -T$,
\be
\label{sec5_2}
\|\Phi (t, \theta_{-t} \omega, (u_0, z_0))\|_{H^1\times L^2}^2
=
\| u(0, -t, \omega, u_0)\|_{H^1}^2
+ \| v(0,-t, \omega, v_0  ) \|^2
\le R(\omega).
\ee
 Denote by
\begin{equation}
\label{absorb}
{\tilde{B}}(\omega) = \{ (u,z) \in H^1(\mathbb{R}^3) \times L^2(\mathbb{R}^3):
\| u\|_{H^1 }^2
+ \| z \|^2
\le R(\omega) \}.
\end{equation}
Then \eqref{sec5_2}
shows that
${\tilde{B}} =\{\tilde{B} (\omega)\}_{\omega \in \Omega}$
is a closed random absorbing set
for $\Phi$ in $\mathcal{D}$.
Next, we show  the pullback asymptotic compactness
of $(u,v)$, which is needed to prove
the asymptotic compactness of $\Phi$.

\begin{lem}
\label{lem41}
Assume that $g \in L^2(\mathbb{R}^3)$, $ h  \in  H^1(\mathbb{R}^3)$ and
 \eqref{f1}-\eqref{f3} hold.
 Then  ,
 for $P$-a.e. $\omega \in \Omega$, the sequence
 $\{ ( u(0, -t_n, \omega, u_{0,n} ), v(0, -t_n, \omega, v_{0,n} ) )
 \}$ has a
 convergent  subsequence in $ H^1(\mathbb{R}^3)
\times L^2(\mathbb{R}^3)$ provided
 $t_n \to \infty$ and $(u_{0,n}, v_{0,n}) \in B(\theta_{-t_n} \omega )$ with
$B =\{B(\omega)\}_{\omega \in \Omega } \in \mathcal{D}$.
 \end{lem}

  \begin{proof}
    Since $t_n \to \infty$, it follows from \eqref{sec5_1} that
 there exists
  $N_1=N_1(B, \omega)>0$ such that
   for all $n \ge N_1$,
  \be
  \label{p41_2}
  \| u(0, -t_n, \omega, u_{0,n})\|^2_{H^1}
  + \| v(0, -t_n, \omega, v_{0,n}) \|^2
  \le  R(\omega).
  \ee
  Notice that \eqref{p41_2} implies that
  there exists
  $(\ut, \vt) \in \hone \times \ltwo$ such that,
  up to a subsequence,
  \be
  \label{p41_3}
  (u(0, -t_n, \omega, u_{0,n}), v(0, -t_n, \omega, v_{0,n}) )
  \to  (\ut, \vt) \  \mbox{weakly in} \ \hone \times \ltwo.
  \ee
  By \eqref{p41_3}  we find that
  \be
  \label{p41_5}
  \liminf_{n \to \infty} \| (u(0, -t_n, \omega, u_{0,n} ),
v(0, -t_n, \omega, v_{0,n}) ) \|_{H^1 \times L^2}
  \ge \| (\ut, \vt) \|_{H^1 \times L^2}
  .
  \ee
  Next we prove that
  \eqref{p41_3}is  actually a strong convergence.
  To this end, taking \eqref{p41_5} into account,
  we only need to
  show
  \be
  \label{p41_6}
  \limsup_{n \to \infty}
\| (u(0, -t_n, \omega, u_{0,n} ),
v(0, -t_n, \omega, v_{0,n}) ) \|_{H^1 \times L^2}
  \le \| (\ut, \vt) \|_{H^1 \times L^2}.
  \ee
  We now prove \eqref{p41_6} by the energy equation
  \eqref{ener}.
    It follows from Lemma \ref{lem31} that
  there exists $N_2=N_2(B, \omega)>0$ such that
  for all $n \ge N_2$,
  \be
  \label{p41_10}
  \| u(t, -t_n, \omega, u_{0,n}) \|^2_{\hone}
  +\| v(t , -t_n, \omega, v_{0,n}) \|^2
  \le e^{-\sigma t} R(\omega),
  \ee
  where $ -t_n \le  t \le 0$.
  Given $m>0$, let $N_3 =N_3(m)>0$ be large enough  such that
  $t_n \ge m$ for all $n\ge N_3$.
  Denote by $N_4= \max\{N_2, N_3\}$. Then by
  \eqref{p41_10} we get that,  for all $n \ge N_4$,
  \be
  \label{p41_11_a1}
  \| u(-m, -t_n, \omega, u_{0,n}) \|^2_{\hone}
  +\| v(-m , -t_n, \omega, v_{0,n}) \|^2
  \le e^{ \sigma m} R(\omega).
  \ee
  By a diagonal procedure, we conclude from
  \eqref{p41_11_a1} that there exist
  a sequence $\{\ut_m, \vt_m\}_{m=1}^\infty$ in
  $\hone \times \ltwo$ and
  a subsequence of $\{(t_n, u_{0,n}, v_{0,n})\}_{n=1}^\infty$ (not relabeled) such that
  for every positive integer $m$, when $n \to \infty$,
  \be
  \label{p41_12}
  (u(-m, -t_n, \omega, u_{0,n}),
    v(-m , -t_n, \omega, v_{0,n}) )
    \to
   (\ut_m, \vt_m )
   \quad \mbox{weakly in} \quad \hone \times \ltwo.
  \ee
  Notice that
 \be
 \label{p41_13}
  (u(0, -t_n, \omega, u_{0,n}), v(0, -t_n, \omega, v_{0,n}))
  \ee
  $$
  = (u(0, -m, \omega, u(-m, -t_n, \omega, u_{0,n})),
  v(0, -m, \omega, v(-m, -t_n, \omega, v_{0,n}))),
  $$
  which along \eqref{p41_12} and
   Lemma \ref{weak_cont}
  implies that, for every positive integer $m$, when
  $n \to \infty$,
  \be
  \label{p41_20}
   u(0, -t_n, \omega, u_{0,n})
  \to   u(0, -m, \omega, \ut_m)
  \quad \mbox{weakly in } \quad \hone ,
  \ee
  and
  \be
  \label{p41_21}
   v(0, -t_n, \omega, v_{0,n})
  \to   v(0, -m, \omega, \vt_m)
  \quad \mbox{weakly in } \quad \ltwo.
  \ee
  By \eqref{p41_3}  and
  \eqref{p41_20}-\eqref{p41_21} we find that
  \be
  \label{p41_22}
  \ut = u(0, -m, \omega, \ut_m)
  \quad \mbox{and}
  \quad
 \vt = v(0, -m, \omega, \vt_m).
 \ee
 Applying \eqref{ener} to
 $(u(0, -m, \omega, \ut_m), v(0, -m, \omega, \vt_m))$,
 by \eqref{p41_22} we get
 \be
 \label{euvt}
 E(\ut, \vt)
 =e^{-4 \sigma m} E(\ut_m, \vt_m)
 +\int_{-m}^0 e^{4 \sigma \xi}
 \Psi (u(\xi, -m, \omega, \ut_m), v(\xi, -m, \omega, \vt_m)) d\xi.
 \ee
 Similarly,   applying \eqref{ener} to
 $(u(0, -m, \omega, u(-m, -t_n, \omega, u_{0,n})  ), v(0, -m, \omega,
  v(-m, -t_n, \omega, v_{0,n} ) ))$,
 by \eqref{p41_13} and \eqref{ener2} we  have
$$
 E ( u(0, -t_n, \omega, u_{0,n} ), v(0, -t_n, \omega, v_{0,n} ))
 $$
 $$
 =e^{-4 \sigma m} E ( u(-m, -t_n, \omega, u_{0,n} ), v(-m, -t_n, \omega, v_{0,n} ))
 $$
$$
 +\int_{-m}^0 e^{4 \sigma \xi}
 \Psi (u(\xi, -m, \omega,u(-m, -t_n, \omega, u_{0,n})  ),
 v(\xi, -m, \omega,  v(-m, -t_n, \omega, v_{0,n} ))) d\xi
$$
$$
=e^{-4 \sigma m} E ( u(-m, -t_n, \omega, u_{0,n} ), v(-m, -t_n, \omega, v_{0,n} ))
$$
$$
  -2(\alpha -\delta -2 \sigma)
  \int^0_{-m} e^{4 \sigma \xi}\| v(\xi, -m, \omega, v(-m, -t_n, \omega, v_{0,n})) \|^2 d\xi
  $$
  $$
  -2( \delta -2 \sigma)(\lambda +\delta^2 -\alpha \delta)
  \int^0_{-m} e^{4 \sigma \xi}\| u(\xi, -m, \omega, u(-m, -t_n, \omega, u_{0,n})) \|^2 d\xi
  $$
  $$
  -2( \delta -2 \sigma)
  \int^0_{-m} e^{4 \sigma \xi}\| \nabla u(\xi, -m, \omega, u(-m, -t_n, \omega, u_{0,n})) \|^2 d\xi
  $$
  $$
  + 8 \sigma
  \int^0_{-m} e^{4 \sigma \xi} \int_{\R^3} F(x,
   u(\xi, -m, \omega, u(-m, -t_n, \omega, u_{0,n} )) ) dx  d\xi
   $$
  $$
  -2 \delta
  \int^0_{-m} e^{4 \sigma \xi} \int_{\R^3}
  u(\xi, -m, \omega, u(-m, -t_n, \omega, u_{0,n} ))
  \times f(x,
   u(\xi, -m, \omega, u(-m, -t_n, \omega, u_{0,n} )) )
    dx   d\xi
   $$
   $$
   + 2(\lambda +\delta^2 -\alpha \delta)
   \int^0_{-m} e^{4 \sigma \xi} \int_{\R^3}
   h(x) u(\xi, -m, \omega, u(-m, -t_n, \omega, u_{0,n} ))
   \omega (\xi) dx d\xi
   $$
    $$
   + 2
   \int^0_{-m} e^{4 \sigma \xi} \int_{\R^3}
   \nabla h(x) \cdot \nabla  u(\xi, -m, \omega, u(-m, -t_n, \omega, u_{0,n} ))
   \omega (\xi) dx d\xi
   $$
   $$
   + 2
   \int^0_{-m} e^{4 \sigma \xi} \int_{\R^3}
   h(x) f(x, u(\xi, -m, \omega, u(-m, -t_n, \omega, u_{0,n} )))
   \omega (\xi) dx d\xi
   $$
   $$
   + 2
   \int^0_{-m} e^{4 \sigma \xi} \int_{\R^3}
   g(x)    v(\xi, -m, \omega, v(-m, -t_n, \omega, v_{0,n} ))
     dx d\xi
   $$
  \be
  \label{euv}
   + 2 (\delta -\alpha)
   \int^0_{-m} e^{4 \sigma \xi} \int_{\R^3}
     h(x)   v(\xi, -m, \omega, v(-m, -t_n, \omega, v_{0,n} ))
   \omega (\xi) dx d\xi.
  \ee
  Now, we need to deal with every term on the right-hand side
  of
  \eqref{euv}. For the first term, by \eqref{ener1} we have
  $$
   e^{-4 \sigma m} E ( u(-m, -t_n, \omega, u_{0,n} ), v(-m, -t_n, \omega, v_{0,n} ))
  $$
  $$
  = e^{-4 \sigma m} \left (\| v(-m, -t_n, \omega, v_{0,n} )\|^2
  + (\lambda + \delta^2 -\alpha \delta)
  \| u(-m, -t_n, \omega, u_{0,n} ) \|^2
  \right )
  $$
  $$
  +  e^{-4 \sigma m} \left (
  \|\nabla  u(-m, -t_n, \omega, u_{0,n} ) \|^2
  +
  2\int_{\R^3} F(x,   u(-m, -t_n, \omega, u_{0,n} ) ) dx
  \right  ),
  $$
  which along with
  \eqref{p41_11_a1} shows that for all $n \ge N_4$,
  $$
   e^{-4 \sigma m} E ( u(-m, -t_n, \omega, u_{0,n} ), v(-m, -t_n, \omega, v_{0,n} ))
   $$
  \be
\label{p41_30}
   \le c e^{-3 \sigma m} R(\omega)
   + 2 e^{-4 \sigma m}\int_{\R^3} F(x,   u(-m, -t_n, \omega, u_{0,n} ) ) dx .
  \ee
  Using \eqref{F3} to estimate the last term on the
  right-hand side of the above,  since $\gamma  \le 3$ we get
  for all $n \ge N_4$,
  $$
    \int_{\R^3} F(x,   u(-m, -t_n, \omega, u_{0,n} ) ) dx
   $$
   $$
   \le  c
   \left (
    \|u(-m, -t_n, \omega, u_{0,n} ) \|^2
    + \| u(-m, -t_n, \omega, u_{0,n} )\|^{\gamma +1}_{\gamma +1}
    + 1
   \right )
   $$
   $$
   \le
   c
   \left (
    \|u(-m, -t_n, \omega, u_{0,n} ) \|^2
    + \| u(-m, -t_n, \omega, u_{0,n} )\|^{\gamma +1}_{H^1}
    + 1
   \right ),
   $$
   which along with \eqref{p41_11_a1} implies that for all
   $n \ge N_4$,
   \be
\label{p41_31}
    \int_{\R^3} F(x,   u(-m, -t_n, \omega, u_{0,n} ) ) dx
   \le
   c \left ( e^{\sigma m} R(\omega) + e^{2\sigma m}R^2(\omega) +1
   \right ).
  \ee
By \eqref{p41_30}-\eqref{p41_31} we get that, for all $n \ge N_4$,
 \be
\label{p41_33}
   e^{-4 \sigma m} E ( u(-m, -t_n, \omega, u_{0,n} ), v(-m, -t_n, \omega, v_{0,n} ))
 \le c e^{-2\sigma m}   ( 1+ R^2 (\omega)   ).
\ee
Next, we deal with  the second term on the right-hand side of
\eqref{euv}.  By \eqref{p41_12} and Lemma \ref{weak_cont}  we find that
for every $\xi \in [-m, 0]$, when $n \to \infty$,
$$
v(\xi, -m, \omega, v(-m, -t_n, \omega, v_{0,n}))
\to v(\xi, -m, \omega, \vt_m)
\quad \mbox{in} \quad  \ltwo,
$$
which implies that, for all $\xi \in [-m,0]$,
\be
\label{p41_34}
\liminf_{n \to \infty}
\| v(\xi, -m, \omega, v(-m, -t_n, \omega, v_{0,n})) \|^2
\ge \|v(\xi, -m, \omega, \vt_m) \|^2.
\ee
By \eqref{p41_34} and Fatou's lemma we obtain
$$
\liminf_{n \to \infty}
\int^0_{-m} e^{4 \sigma \xi}
\|   v(\xi, -m, \omega, v(-m, -t_n, \omega, v_{0,n})) \|^2 d\xi
$$
$$
\ge
\int^0_{-m} e^{4 \sigma \xi}
\liminf_{n \to \infty}
\|   v(\xi, -m, \omega, v(-m, -t_n, \omega, v_{0,n})) \|^2 d\xi
$$
$$
\ge
\int^0_{-m} e^{4 \sigma \xi}
\|   v(\xi, -m, \omega,  \vt_m ) \|^2 d\xi.
$$
Therefore,  by \eqref{kappa} we have
$$
\limsup_{n \to \infty} -2(\alpha -\delta -2\sigma)
\int^0_{-m} e^{4 \sigma \xi}
\|   v(\xi, -m, \omega, v(-m, -t_n, \omega, v_{0,n})) \|^2 d\xi
$$
$$
= -2(\alpha -\delta -2\sigma)
\liminf_{n \to \infty}
\int^0_{-m} e^{4 \sigma \xi}
\|   v(\xi, -m, \omega, v(-m, -t_n, \omega, v_{0,n})) \|^2 d\xi
$$
\be
\label{p41_35}
\le
-2(\alpha -\delta -2\sigma)
\int^0_{-m} e^{4 \sigma \xi}
\|   v(\xi, -m, \omega,  \vt_m ) \|^2 d\xi.
\ee
Similarly, by \eqref{delta}, \eqref{kappa},  \eqref{p41_12} and
Fatou's lemma, we can also  prove that
$$
\limsup_{n\to\infty}
  -2( \delta -2 \sigma)(\lambda +\delta^2 -\alpha \delta)
  \int^0_{-m} e^{4 \sigma \xi}\| u(\xi, -m, \omega, u(-m, -t_n, \omega, u_{0,n})) \|^2 d\xi
  $$
\be
\label{p41_36}
\le
  -2( \delta -2 \sigma)(\lambda +\delta^2 -\alpha \delta)
\int^0_{-m} e^{4 \sigma \xi}
\|   u(\xi, -m, \omega,  \ut_m ) \|^2 d\xi,
\ee
and
  $$
\limsup_{n\to\infty}
  -2( \delta -2 \sigma)
  \int^0_{-m} e^{4 \sigma \xi}\| \nabla u(\xi, -m, \omega, u(-m, -t_n, \omega, u_{0,n})) \|^2 d\xi
  $$
\be
\label{p41_37}
\le
   -2( \delta -2 \sigma)
\int^0_{-m} e^{4 \sigma \xi}
\|  \nabla u(\xi, -m, \omega,  \ut_m ) \|^2 d\xi.
\ee
Next, we prove the convergence of the fifth term on the right-hand side of
\eqref{euv} which is a nonlinear term. We claim
$$
\lim_{n \to \infty} \int_{-m}^0 e^{4\sigma \xi}\int_{\R^3} F(x, u(\xi, -m, \omega, u(-m, -t_n, \omega, u_{0,n} ))) dxd\xi
$$
\be
\label{p41_38}
= \int_{-m}^0 e^{4\sigma \xi}\int_{\R^3} F(x, u(\xi, -m, \omega, \ut_m )) dxd\xi.
\ee
To prove \eqref{p41_38} we write
$$
 | \int_{-m}^0 e^{4\sigma \xi}\int_{\R^3} \left ( F(x, u(\xi, -m, \omega, u(-m, -t_n, \omega, u_{0,n} )))
-
   F(x, u(\xi, -m, \omega, \ut_m )) \right ) dxd\xi|
$$
 $$
\le
  \int_{-m}^0 e^{4\sigma \xi}\int_{|x|> k} | F(x, u(\xi, -m, \omega, u(-m, -t_n, \omega, u_{0,n} )))
-
   F(x, u(\xi, -m, \omega, \ut_m )) | dxd\xi
$$
\be
\label{p41_39}
+
  |\int_{-m}^0 e^{4\sigma \xi}\int_{|x| <k}  F(x, u(\xi, -m, \omega, u(-m, -t_n, \omega, u_{0,n} )))
-
   F(x, u(\xi, -m, \omega, \ut_m ))   dxd\xi |.
\ee
Given $\epsilon>0$, by Lemma \ref{lem32} we find that there are
$k_1 =k_1(\omega, \epsilon)>0$ and $N_5 =N_5(B, \omega, \epsilon) \ge N_4$
such that for all $k \ge k_1$ and $n\ge N_5$,
\be
\label{p41_40}
\int_{|x| >k} |u(\xi, -t_n, \omega, u_{0,n})|^2 dx
\le \epsilon e^{-\sigma \xi},
\ee
where $\xi \in [-t_n, 0]$.
Hence,  by \eqref{F3} we obtain
that for all $k \ge k_1$ and $n\ge N_5$,
$$
\int_{|x| >k}  | F(x, u(\xi,   -t_n, \omega, u_{0,n} ))|dx
$$
$$
\le
\int_{|x| >k}  \left ( |u(\xi,   -t_n, \omega, u_{0,n} ) | ^2
+ |u(\xi,   -t_n, \omega, u_{0,n} ) | ^{\gamma +1} + \phi_1^2 + \phi_2
\right ) dx
$$
$$
\le
\int_{|x| >k} (\phi_1^2 + \phi_2) dx
+
\int_{|x| >k}
 |u(\xi,   -t_n, \omega, u_{0,n} ) | ^2 dx
$$
$$
+\left (
\int_{|x| >k} |u(\xi,   -t_n, \omega, u_{0,n} ) |^{2\gamma } dx
\right )^{\frac 12}
\left (
\int_{|x| >k} |u(\xi,   -t_n, \omega, u_{0,n} ) |^{2 } dx
\right )^{\frac 12}
$$
$$
\le
\int_{|x| >k} (\phi_1^2 + \phi_2) dx
+ \epsilon e^{-\sigma \xi}
+ \sqrt{\epsilon} e^{-{\frac \sigma{2}} \xi}
 \left (
\int_{\R^3} |u(\xi,   -t_n, \omega, u_{0,n} ) |^{2\gamma } dx
\right )^{\frac 12}
$$
$$
\le
\int_{|x| >k} (\phi_1^2 + \phi_2) dx
+ \epsilon e^{-\sigma \xi}
+ \sqrt{\epsilon} e^{-{\frac \sigma{2}} \xi}
  \|u(\xi,   -t_n, \omega, u_{0,n} ) \|_{H^1}^\gamma,
$$
which along with the fact  $ \gamma \le 3$ and  \eqref{p41_10}
implies that
\be
\label{p41_40_a1}
\int_{|x| >k}  | F(x, u(\xi,   -t_n, \omega, u_{0,n} ))|dx
\le
\int_{|x| >k} (\phi_1^2 + \phi_2) dx
+ \epsilon e^{-\sigma \xi}
+ c \sqrt{\epsilon} e^{-{\frac \sigma{2}} \xi}
  (1+  e^{- {\frac{3\sigma}2} \xi} R^{\frac 32} (\omega)).
\ee
Notice that there is $k_2 =k_2(\epsilon)>0$ such that
for all $k \ge k_2$,
the first term on the right-hand side of
\eqref{p41_40_a1} is bounded by $\epsilon$.
Therefore, for all $\xi \le 0$,  $n \ge N_5$ and $k \ge k_3 =\max \{k_1, k_2\}$,
\be
\label{p41_40_a2}
\int_{|x| >k}  | F(x, u(\xi,   -t_n, \omega, u_{0,n} ))|dx
\le
\epsilon
+ e^{-2 \sigma \xi}
\left ( \epsilon + \sqrt{\epsilon} c
+\sqrt{\epsilon}\;  c \; R ^{\frac 32} (\omega) \right ).
\ee
On the other hand, there exits
$k_4=k_4(m, \omega, \epsilon) \ge k_3$ such that for all
$k \ge k_4$,
\be
\label{p41_41}
  \int_{-m}^0 e^{4\sigma \xi}\int_{|x|> k} |
   F(x, u(\xi, -m, \omega, \ut_m )) | dxd\xi
\le \epsilon.
\ee
By \eqref{p41_40_a2}-\eqref{p41_41}, the first term on the right-hand
side of \eqref{p41_39} satisfies, for all $n \ge N_5$ and
$k \ge k_4$,
$$
  \int_{-m}^0 e^{4\sigma \xi}\int_{|x|> k} | F(x, u(\xi, -m, \omega, u(-m, -t_n, \omega, u_{0,n} )))
-
   F(x, u(\xi, -m, \omega, \ut_m )) | dxd\xi
$$
$$
=
  \int_{-m}^0 e^{4\sigma \xi}\int_{|x|> k} | F(x, u(\xi,   -t_n, \omega, u_{0,n} ))
-
   F(x, u(\xi, -m, \omega, \ut_m )) | dxd\xi
$$
$$
\le
  \int_{-m}^0 e^{4\sigma \xi}\int_{|x|> k} | F(x, u(\xi,   -t_n, \omega, u_{0,n} ))|dx d\xi
$$
$$
+
\int_{-m}^0 e^{4\sigma \xi}\int_{|x|> k}
   |F(x, u(\xi, -m, \omega, \ut_m )) | dxd\xi
$$
$$
\le \epsilon
 +\epsilon  \int_{-m}^0 e^{4\sigma \xi}
 d\xi
+ (\epsilon + \sqrt{\epsilon} c
+ \sqrt{\epsilon}\; c \;  R^{\frac 32} (\omega) )
\int^0_{-m} e^{2 \sigma \xi} d\xi
$$
\be
\label{p41_42}
\le   \sqrt{\epsilon}\; c (1
+   R^{\frac 32} (\omega)  )
\quad \mbox{for all} \quad \epsilon \le 1.
\ee
To deal with the second term on the right-hand side
of \eqref{p41_39}, we notice that,
  by
\eqref{p41_12} and Lemma \ref{weak_cont},
 when $n \to \infty$,
\be
\label{p41_43}
u(\xi, -m, \omega, u(-m, -t_n, \omega, u_{0,n}))
\to u(\xi, -m, \omega, \ut_m) \quad \mbox{ weakly
in} \    \hone ,
\ee
for  $\xi \in [-m, 0]$.
By  \eqref{p41_43} and the compactness of embedding
$H^1(Q_k) \hookrightarrow
L^2(Q_k)$, we find that, for $\xi \in [-m, 0]$,
\be
\label{p41_44}
u(\xi, -m, \omega, u(-m, -t_n, \omega, u_{0,n}))
\to u(\xi, -m, \omega, \ut_m) \quad \mbox{ strongly
in} \   L^2(Q_k) .
\ee
We also have
$$
 |\int_{|x|<k}  ( F(x, u(\xi, -m, \omega, u(-m, -t_n, \omega, u_{0,n} )) )
-
   F(x, u(\xi, -m, \omega, \ut_m ) ) )   dx |
$$
$$
 = |\int_{|x|<k}  {\frac {\partial F}{\partial u}}(x, \bar{u}) \
 (   u(\xi, -m, \omega, u(-m, -t_n, \omega, u_{0,n} ) )
-  u(\xi, -m, \omega, \ut_m  ) )  dx |
$$
$$
 = |\int_{|x|<k}  f(x, \bar{u}) \
 (   u(\xi, -m, \omega, u(-m, -t_n, \omega, u_{0,n} ) )
-  u(\xi, -m, \omega, \ut_m  ) )  dx |
$$
\be
\label{p41_50}
 \le \left (  \int_{\R^3} | f(x, \bar{u})|^2 dx \right )^{\frac 12}
 \|  u(\xi, -m, \omega, u(-m, -t_n, \omega, u_{0,n} ) )
-  u(\xi, -m, \omega, \ut_m  ) ) \|_{L^2(Q_k)}.
\ee
By \eqref{f1}  and \eqref{p41_10} we get
$$
\left (  \int_{\R^3} | f(x, \bar{u})|^2 dx \right )^{\frac 12}
$$
$$
\le c \left (
  \|u(\xi, -m, \omega, u(-m, -t_n, \omega, u_{0,n} ) )\|_{H^1}^\gamma
+ \|  u(\xi, -m, \omega,\ut_m )\|^\gamma + \|\phi_1 \|^2
\right )
$$
 $$
\le c \left (
  \|u(\xi,   -t_n, \omega, u_{0,n}  )\|_{H^1}^\gamma
+ \|  u(\xi, -m, \omega,\ut_m )\|^\gamma + \|\phi_1 \|^2
\right )
$$
$$
\le c \left ( e^{-{\frac {\sigma  \gamma}{2}} \xi} R^{\frac \gamma2} (\omega)
+ \|  u(\xi, -m, \omega,\ut_m )\|^\gamma + \|\phi_1 \|^2
\right ),
$$
which along with \eqref{p41_44} and \eqref{p41_50} implies that,
as $n \to \infty$,
\be
\label{p41_51}
 \int_{|x|<k}    F(x, u(\xi, -m, \omega, u(-m, -t_n, \omega, u_{0,n} ) )) dx
\to
  \int_{|x|<k}  F(x, u(\xi, -m, \omega, \ut_m )  )  dx  .
\ee
It follows from
\eqref{p41_10}, \eqref{p41_51} and the dominated convergence
theorem that, when $n \to \infty$,
$$
 \int^0_{-m} e^{4\sigma \xi} \int_{|x|<k}    F(x, u(\xi, -m, \omega, u(-m, -t_n, \omega, u_{0,n} ) ) ) dx d\xi
$$
$$
\to
  \int^0_{-m} e^{4\sigma \xi}
\int_{|x|<k}  F(x, u(\xi, -m, \omega, \ut_m )   )  dx d\xi  .
$$
Therefore, there exists $N_6 \ge N_5$  such that for all
$n \ge N_6$,
$$
 | \int^0_{-m} e^{4\sigma \xi} \int_{|x|<k}
  \left ( F(x, u(\xi, -m, \omega, u(-m, -t_n, \omega, u_{0,n} ) ))
-
     F(x, u(\xi, -m, \omega, \ut_m )   ) \right ) dx d\xi |
\le \epsilon  ,
$$
which along with \eqref{p41_39} and \eqref{p41_42}
implies  \eqref{p41_38}.
By an   argument  similar to the proof
of
\eqref{p41_38}, we can also show the convergence of the
sixth term on the right-hand side of \eqref{euv} (details are omitted).
That is, we have that, as $n \to \infty$,
 $$
\int^0_{-m} e^{4 \sigma \xi} \int_{\R^3}
  u(\xi, -m, \omega, u(-m, -t_n, \omega, u_{0,n}  ))
  \times f(x,
   u(\xi, -m, \omega, u(-m, -t_n, \omega, u_{0,n} )))
    dx   d\xi
   $$
  \be
\label{p41_60} \to
\int^0_{-m} e^{4 \sigma \xi} \int_{\R^3}
  u(\xi, -m, \omega, \ut_m  )
  \times f(x,
   u(\xi, -m, \omega, \ut_m ))
    dx   d\xi.
\ee
The convergence of the remaining terms on the right-hand side of
\eqref{euv} is given below, which  can be proved by a similar (actually  simpler) procedure.
$$
   \int^0_{-m} e^{4 \sigma \xi} \int_{\R^3}
   h(x) u(\xi, -m, \omega, u(-m, -t_n, \omega, u_{0,n} ))
   \omega (\xi) dx d\xi
   $$
\be
\label{p41_70}
\to
   \int^0_{-m} e^{4 \sigma \xi} \int_{\R^3}
   h(x) u(\xi, -m, \omega, \ut_m   )
   \omega (\xi) dx d\xi.
  \ee
   $$
   \int^0_{-m} e^{4 \sigma \xi} \int_{\R^3}
   \nabla h(x) \cdot \nabla  u(\xi, -m, \omega, u(-m, -t_n, \omega, u_{0,n} ))
   \omega (\xi) dx d\xi
   $$
  \be
\label{p41_71}
 \to
   \int^0_{-m} e^{4 \sigma \xi} \int_{\R^3}
   \nabla h(x) \cdot \nabla  u(\xi, -m, \omega, \ut_m   )
   \omega (\xi) dx d\xi.
   \ee
   $$
   \int^0_{-m} e^{4 \sigma \xi} \int_{\R^3}
   h(x) f(x, u(\xi, -m, \omega, u(-m, -t_n, \omega, u_{0,n} )))
   \omega (\xi) dx d\xi
   $$
\be
\label{p41_73}
\to
   \int^0_{-m} e^{4 \sigma \xi} \int_{\R^3}
   h(x) f(x, u(\xi, -m, \omega, \ut_m ))
   \omega (\xi) dx d\xi.
  \ee
   $$
   \int^0_{-m} e^{4 \sigma \xi} \int_{\R^3}
   g(x)    v(\xi, -m, \omega, v(-m, -t_n, \omega, v_{0,n} ))
     dx d\xi
   $$
 \be
\label{p41_76}
   \to \int^0_{-m} e^{4 \sigma \xi} \int_{\R^3}
   g(x)    v(\xi, -m, \omega, \vt_m  )
     dx d\xi.
  \ee
$$
   \int^0_{-m} e^{4 \sigma \xi} \int_{\R^3}
     h(x)   v(\xi, -m, \omega, v(-m, -t_n, \omega, v_{0,n} ))
   \omega (\xi) dx d\xi$$
\be
\label{p41_77}
   \to  \int^0_{-m} e^{4 \sigma \xi} \int_{\R^3}
     h(x)   v(\xi, -m, \omega, \vt_m )
   \omega (\xi) dx d\xi.
\ee
Now, taking  the limit of \eqref{euv} as $n \to \infty$,
by \eqref{p41_33}, \eqref{p41_35}-\eqref{p41_38}
and \eqref{p41_60}-\eqref{p41_77} we find that
$$
\limsup_{n \to \infty} E(u(0, -t_n, \omega, u_{0,n}), v(0, -t_n, \omega, v_{0,n}))
$$
$$
\le
c e^{-2 \sigma m} (1 + R ^2(\omega))
  -2(\alpha -\delta -2 \sigma)
  \int^0_{-m} e^{4 \sigma \xi}\| v(\xi, -m, \omega, \vt_m ) \|^2 d\xi
  $$
  $$
  -2( \delta -2 \sigma)(\lambda +\delta^2 -\alpha \delta)
  \int^0_{-m} e^{4 \sigma \xi}\| u(\xi, -m, \omega, \ut_m ) \|^2 d\xi
  $$
  $$
  -2( \delta -2 \sigma)
  \int^0_{-m} e^{4 \sigma \xi}\| \nabla u(\xi, -m, \omega, \ut_m ) \|^2 d\xi
  $$
  $$
  + 8 \sigma
  \int^0_{-m} e^{4 \sigma \xi} \int_{\R^3} F(x,
   u(\xi, -m, \omega, \ut_m ) ) dx  d\xi
   $$
  $$
  -2 \delta
  \int^0_{-m} e^{4 \sigma \xi} \int_{\R^3}
  u(\xi, -m, \omega,\ut_m )
  \times f(x,
   u(\xi, -m, \omega, \ut_m ))
    dx   d\xi
   $$
   $$
   + 2(\lambda +\delta^2 -\alpha \delta)
   \int^0_{-m} e^{4 \sigma \xi} \int_{\R^3}
   h(x) u(\xi, -m, \omega, \ut_m )
   \omega (\xi) dx d\xi
   $$
    $$
   + 2
   \int^0_{-m} e^{4 \sigma \xi} \int_{\R^3}
   \nabla h(x) \cdot \nabla  u(\xi, -m, \omega, \ut_m )
   \omega (\xi) dx d\xi
   $$
   $$
   + 2
   \int^0_{-m} e^{4 \sigma \xi} \int_{\R^3}
   h(x) f(x, u(\xi, -m, \omega, \ut_m ))
   \omega (\xi) dx d\xi
   $$
   $$
   + 2
   \int^0_{-m} e^{4 \sigma \xi} \int_{\R^3}
   g(x)    v(\xi, -m, \omega, \vt_m )
     dx d\xi
   $$
  \be
  \label{p41_80}
   + 2 (\delta -\alpha)
   \int^0_{-m} e^{4 \sigma \xi} \int_{\R^3}
     h(x)   v(\xi, -m, \omega, \vt_m )
   \omega (\xi) dx d\xi.
  \ee
It follows from \eqref{ener2} and \eqref{p41_80} that
$$
\limsup_{n \to \infty} E(u(0, -t_n, \omega, u_{0,n}), v(0, -t_n, \omega, v_{0,n}))
$$
\be
\label{p41_81}
\le
c e^{-2 \sigma m} (1 + R ^2(\omega))
  + \int_{-m}^0 e^{4 \sigma \xi} \Psi(u(\xi, -m, \omega, \ut_m  ),
v(\xi, -m, \omega, \vt_m  )) d \xi.
\ee
By \eqref{euvt} and \eqref{p41_81} we find that
$$
\limsup_{n \to \infty} E(u(0, -t_n, \omega, u_{0,n}), v(0, -t_n, \omega, v_{0,n}))
$$
\be
\label{p41_82}
\le
c e^{-2 \sigma m} (1 + R ^2(\omega))
  -e^{-4 \sigma m} E(\ut_m, \vt_m )
+ E(\ut, \vt).
\ee
For the second term on the right-hand side
of \eqref{p41_82}, by \eqref{ener1} and \eqref{F2} we have
\be
\label{p41_83}
-e^{-4 \sigma m} E(\ut_m, \vt_m )
\le 2 e^{-4 \sigma m} \int_{\R^3} \phi_3 (x) dx.
\ee
It follows from \eqref{p41_82}-\eqref{p41_83} that
 $$
\limsup_{n \to \infty} E(u(0, -t_n, \omega, u_{0,n}), v(0, -t_n, \omega, v_{0,n}))
$$
\be
\label{p41_84}
\le
c e^{-2 \sigma m} (1 + R ^2(\omega))
  +2 e^{-4 \sigma m} \int_{\R^3} \phi_3 (x) dx
+ E(\ut, \vt).
\ee
 Let $m \to \infty$. Then we get that
 \be
\label{p41_90}
\limsup_{n \to \infty} E(u(0, -t_n, \omega, u_{0,n}), v(0, -t_n, \omega, v_{0,n}))
 \le E(\ut, \vt) .
\ee
On the other hand, It follows from \eqref{p41_40_a2} and
\eqref{p41_51} with $\xi =0$ that, as $n \to \infty$,
$$
\int_{\R^3} F(x, u(0, -t_n, \omega, u_{0,n})) dx
\to \int_{\R^3} F(x, \ut) dx,
$$
which
   along with \eqref{ener1} shows that
$$
\limsup_{n \to \infty}
E(u(0, -t_n, \omega, u_{0,n}), v(0, -t_n, \omega, v_{0,n}))
= 2 \int_{\R^3} F(x, \ut) dx
$$
$$
+ \limsup_{n\to \infty}
\left (
\| v(0, -t_n, \omega, v_{0,n})\|^2
+(\lambda + \delta^2 -\alpha \delta) \| u(0, -t_n, \omega, u_{0,n}) \|^2
+\| \nabla u(0, -t_n, \omega, u_{0,n}) \|^2
\right ).
$$
Substituting the above equality into
\eqref{p41_90}, by \eqref{ener1} we obtain that
$$
\limsup_{n\to \infty}
\left (
\| v(0, -t_n, \omega, v_{0,n})\|^2
+(\lambda + \delta^2 -\alpha \delta) \| u(0, -t_n, \omega, u_{0,n}) \|^2
+\| \nabla u(0, -t_n, \omega, u_{0,n}) \|^2
\right )
$$
\be
\label{p41_92}
\le
 \| \vt \|^2
+ (\lambda + \delta^2 -\alpha \delta) \| \ut\|^2
+ \| \nabla \ut \|^2.
\ee
Notice that the left   and right expressions
are equivalent norms of $\hone \times \ltwo$.
Therefore, by \eqref{delta} and \eqref{p41_92} we find that
$$ \limsup_{n\to \infty}
\left (
  \| u(0, -t_n, \omega, u_{0,n}) \|^2_{H^1}
+\| v(0, -t_n, \omega,v_{0,n}) \|^2
\right )
\le
\| \ut \|^2_{H^1} + \| \vt \|^2,
$$
which implies \eqref{p41_6}. Finally, we get the following
strong convergence by \eqref{p41_3}-\eqref{p41_6}:
$$
  (u(0, -t_n, \omega, u_{0,n}), v(0, -t_n, \omega, v_{0,n}) )
  \to  (\ut, \vt) \  \mbox{strongly  in} \ \hone \times \ltwo.
 $$
This completes the proof.
 \end{proof}

As an immediate consequence of Lemma \ref{lem41},
 we see that
the random dynamical system $\Phi$ is pullback
asymptotically compact
in $\hone \times \ltwo$.

\begin{lem}
\label{lem42}
Assume that $g \in L^2(\mathbb{R}^3)$, $ h  \in  H^1(\mathbb{R}^3)$ and
 \eqref{f1}-\eqref{f3} hold.
 Then the random dynamical system $\Phi$ is
 $\mathcal{D}$-pullback  asymptotically compact
 in $ H^1(\mathbb{R}^3)
\times L^2(\mathbb{R}^3)$; that is,
 for $P$-a.e. $\omega \in \Omega$, the sequence
 $\{\Phi (t_n, \theta_{-t_n} \omega,  (u_{0,n}, z_{0,n} ))\}$ has a
 convergent  subsequence in $ H^1(\mathbb{R}^3)
\times L^2(\mathbb{R}^3)$ provided
 $t_n \to \infty$ and $(u_{0,n}, z_{0,n}) \in B(\theta_{-t_n} \omega )$ with
$B =\{B(\omega)\}_{\omega \in \Omega } \in \mathcal{D}$.
 \end{lem}

We are now in a position to prove existence of a random
attractor for the stochastic wave equation.

\begin{thm}
Assume that $g \in L^2(\mathbb{R}^3)$, $ h  \in  H^1(\mathbb{R}^3)$ and
 \eqref{f1}-\eqref{f3} hold.
   Then the random dynamical
 system $\Phi$ has a unique $\mathcal{D}$-random
 attractor  $\{\mathcal{A}(\omega)\}_{\omega \in \Omega}$
 in $H^1 (\mathbb{R}^3) \times L^2(\mathbb{R}^3)$.
 \end{thm}

\begin{proof}
Notice that $\Phi$ has a closed  absorbing set
$ \tilde{B} = \{\tilde{B}(\omega)\}_{\omega \in \Omega}$ in $\mathcal{D}$ by \eqref{sec5_2}-\eqref{absorb}, and is $\mathcal{D}$-pullback asymptotically compact
in $\hone \times \ltwo$ by Lemma \ref{lem42}. Hence the existence of a unique
$\mathcal{D}$-random attractor  follows from Proposition
\ref{att}
 immediately.
\end{proof}

\end{document}